\documentclass[11pt, dvipsnames]{amsart}
\usepackage[utf8]{inputenc}
\usepackage{subcaption}
\usepackage{graphicx}
\usepackage{graphicx}
\usepackage{grffile}
\usepackage{a4wide}
\usepackage{color}
\usepackage{xcolor}
\usepackage{amsmath,amssymb}
\usepackage{float}
\usepackage{soul}
\usepackage{lipsum}
\usepackage{hyperref}
\usepackage{cleveref}

\usepackage{pgfplots}
\usepgfplotslibrary{fillbetween}
\pgfplotsset{compat=newest}
\usepackage{pgfplotstable}
\usepgfplotslibrary{groupplots}
\usepackage{diagbox}
\usepackage{overpic}
\usepackage{booktabs}  
\usepackage{capt-of}
\usepackage{multirow}
\usepackage[mathlines]{lineno}
\usepackage{commath}

\usepackage[foot]{amsaddr}
\usepackage{float}
\usepackage{amsthm}
\usepackage{commath}
\usepackage{ulem}

\newtheorem{theorem}{Theorem}[section]
\newtheorem{lemma}[theorem]{Lemma}
\newtheorem{example}{Example}

\newtheorem{remark}{Remark}[section]

\numberwithin{equation}{section} 

\crefname{section}{Section}{Sections}
\crefname{subsection}{Section}{Sections}
\crefname{subsubsection}{Section}{Sections}
\crefname{example}{Example}{Examples}

\newcommand{\bm}[1]{\boldsymbol{#1}} 
\newcommand{\enorm}[1]{{\left\vert\kern-0.25ex\left\vert\kern-0.25ex\left\vert #1 
    \right\vert\kern-0.25ex\right\vert\kern-0.25ex\right\vert}}
 \newcommand{\vertiii}[1]{{\left\vert\kern-0.25ex\left\vert\kern-0.25ex\left\vert #1 
    \right\vert\kern-0.25ex\right\vert\kern-0.25ex\right\vert}}%

\newcommand{\mesh}{\mathcal{T}_h}

\newcommand{\DG}{\mathrm{DG}}

\newcommand{\units}[2]{\ensuremath{#1\,\mathrm{#2}}}

\title[DG for the EMI model]{
A Discontinuous Galerkin Method for the Extracellular Membrane Intracellular Model}

\author{Rami Masri$^1{}^{,2}$, Keegan L. A. Kirk$^3$, Eirill Hauge$^1$, and Miroslav Kuchta$^1$}

\email{rami\_masri@brown.edu, kkirk6@gmu.edu, eirill@simula.no,  miroslav@simula.no}
\address{$^1$Department of Numerical Analysis and Scientific Computing, Simula Research Laboratory, Oslo, 0164 Norway. RM and EH gratefully acknowledge support from the Research Council of Norway (RCN) via FRIPRO grant agreement 324239 (EMIx).
MK gratefully acknowledges support from the RCN grant 303362.
}
\address{$^2$Division of Applied Mathematics, Brown University, Providence, RI 02906.}
\address{$^3$Department of Mathematical Sciences and Center for Mathematics and Artificial Intelligence, George Mason University, Fairfax, VA 22030, USA. KK gratefully acknowledges support from the Natural Sciences and Engineering Research Council of Canada (NSERC) via PDF-568008.}
\date{ \today }
\usetikzlibrary{shapes, patterns, arrows}
\usetikzlibrary{calc}
\begin{document}
\begin{abstract}
We formulate and analyze interior penalty discontinuous Galerkin methods for coupled elliptic PDEs modeling excitable tissue, represented by intracellular and extracellular domains sharing a common interface. The PDEs are coupled through a dynamic boundary condition, posed on the interface, that relates the normal gradients of the solutions to the time derivative of their jump. This system is referred to as the Extracellular Membrane Intracellular model or the cell-by-cell model. Due to the dynamic nature of the interface condition and to the presence of corner singularities, the analysis of discontinuous Galerkin methods is non-standard. We prove the existence and uniqueness of solutions by a reformulation of the problem to one posed on the membrane. Convergence is shown by utilizing face-to-element lifting operators and notions of weak consistency suitable for solutions with low spatial regularity. Further, we present parameter-robust preconditioned iterative solvers. Numerical examples in idealized geometries demonstrate our theoretical findings, and simulations in multiple cells portray the robustness of the method.

 \vspace{1em}
  
\end{abstract}

\newtheorem{thm}{Theorem}[section]

\maketitle

\section{Introduction}\label{sec:intro}

Electrophysiology models of excitable cells (such as cardiac cells or neurons) clasically consisted of the monodomain or bidomain equations ~\cite[Ch 2.]{sundnes2007computing}, ~\cite{franzone2002degenerate}.
  These equations 
 homogenize the tissue and assume coexistence of the extracellular, intracellular spaces and the cell membranes in every spatial point. In contrast, the geometry of each of these three domains is explicitly represented in the  Extracellular Membrane Intracellular (EMI) system. In particular, this system models the electric potentials in the cell and its surrounding along with the trans--membrane potential. 

Consequently, the EMI-type models can naturally incorporate highly detailed reconstructions of biological tissues, e.g. ~\cite{schlegel2024whole}, and numerical simulations may be used to shed light on previously unobservable physiological phenomena. 
For example,
simulations of neuronal networks often exclusively model the membrane of the cells, connected through synapses. However, network behaviour may also be influenced by indirect cell communication through the extracellular space~\cite{anastassiouephaptic2011}. This \textit{ephaptic coupling} can be captured using EMI-type models.
We refer to the book~\cite{tveito2021modeling} for more details, and to~\cite{amar2005existence,veneroni2006reaction} for the analysis of the EMI system. 

The EMI model is commonly discretized by the $H^1$-conforming continuous Galerkin (CG) elements,
see e.g.~\cite{agudelo2013computationally, tveito2017cell, huynh2023convergence}. The analysis of the
CG formulation has been established in~\cite{fokoue2023numerical}.  Recently, there has been a
growing interest in using discontinuous Galerkin (DG) methods  
for the monodomain and bidomain equations~\cite{botta2024high},
(Poisson)-Nernst-Planck equations modeling ionic transport~\cite{liu2017free, liu2022positivity, roy2023scalable} and the related EMI models of electrodiffusion ~\cite{ellingsrud2024splitting}.

We note that the EMI equations are closely related to the elliptic interface models of porous media flow in domains containing 
blocking fractures/barriers. However, compared to EMI,  the latter models consist of adjacent domains and do not have the dynamic boundary condition on the interface.   For these elliptic interface models,  DG methods have been analyzed,  for example in ~\cite{huang2020high, huynh2013high} under the assumption of smooth solutions and in 
~\cite{cangiani2018adaptive} under minimal regularity assumptions where the authors leverage the so-called medius analysis introduced in \cite{Gudi2010ANE}. Here, we adopt a completely different approach following~\cite{ern2021quasi} to show convergence for solutions with $H^{1+s}$ spatial regularity for $0 < s < 1/2$.

Our motivation for considering the nonconforming discretizations is threefold. First, DG 
schemes provide better local (element-wise) conservation properties. Second, implementation of
CG for EMI requires that the finite element (FE) framework supports multimesh (mixed-dimensional)
features, which might not be readily available. As an example, in the popular open-source
finite element library FEniCS, preliminary support for such features has only recently been added~\cite{carry2021abstractions}.
Here, we argue that discretization by DG allows the EMI model to be
implemented in any FE code that can handle standard discontinuous Galerkin methods. 
Finally, the design of solvers for the linear systems arising from CG discretization
is still an active area of research~\cite{huynh2023convergence, jaeger2021efficient, tveito2021modeling, budivsa2024algebraic}.
In particular,~\cite[Ch 6.]{tveito2021modeling} and ~\cite{budivsa2024algebraic} show that black-box
multigrid solvers can fail to provide order-optimal solvers for certain parameter regimes.
To this end, we shall below explore multigrid methods with DG.

\textbf{Main contributions}. We are mainly concerned with the numerical analysis of the interior penalty DG method formulated for the EMI equations. In particular, the main challenges and contributions are summarized below. 
\begin{itemize}
    \item  We show existence and uniqueness of DG semi-discrete solutions in \Cref{sec:semi_discrete_posed}. Since the time-dependent term appears as an interface condition on the membrane, a reformulation of the model is required and is attained by introducing suitable lifting operators. The main result is in \Cref{thm:exist_unique}. 
    The stability of the semi-discrete DG solution is established in \Cref{sec:stability}.
    \item We prove convergence of the semi-discrete solution under low spatial regularity assumptions where we work with $H^{1+s}$ for $s \in (0,1/2)$ functions with Laplacians in $L^2(\Omega)$. This requires a delicate analysis relying on face-to-element lifting operators and weak notions of consistency established in~\cite{ern2021quasi}. This analysis is presented in \Cref{sec:convergence} and the main result is \Cref{thm:convergence}.
Further,  we present and analyze a backward Euler DG discretization, see \Cref{sec:fullydiscrete}.
    \item Finally, parameter robust preconditioners with respect to both the time step and the length of the domain are proposed. The convergence properties and the robustness of our solver are demonstrated in several examples in \Cref{sec:numerics}, including a physiologically relevant 3D simulation of 15 cardiac cells. 
\end{itemize}

\section{Problem statement} For $d=2, 3$, we let $\Omega \subset\mathbb{R}^d$
be a bounded polygonal domain which contains a subdomain $\Omega_i \subset\mathbb{R}^d$. 
We will refer to $\Omega_e = \Omega \backslash \Omega_i$ as the exterior or extracellular domain while $\Omega_i$ is termed the intracellular domain or cell. Further, let $\Gamma =\partial\Omega_e\cap\partial\Omega_i$
 be the common interface between the
domains, and let $\Gamma_e =\partial\Omega_e \setminus\Gamma$
be the outer boundary of the extracellular domain. See \Cref{fig:illustrate} for an illustration of the geometry. 
  \begin{figure}[H]
    \begin{overpic}[height=0.225\textwidth]{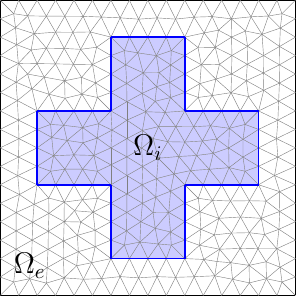}
    \put(65,20){$\Gamma$}
    \put(105,85){$\Gamma_e$}
\end{overpic}
\caption{Intracellular domain $\Omega_i$ (or cell) in blue surrounded by the extracellular domain $\Omega_e$ with interface $\Gamma$ and $\Omega = \Omega_i \cup \Omega_e$. The outer boundary of $\Omega_e$ is denoted by $\Gamma_e$.}
\label{fig:illustrate}
\end{figure}
We consider the EMI model~\cite{franzone2002degenerate,tveito2021modeling} modeling the intracellular and extracellular potentials:
\begin{equation}\label{eq:EMI}
  \begin{alignedat}{2}
    -\nabla\cdot(\kappa_e\nabla u_e) &= f_e   && \quad \text{ in }\Omega_e \times (0,T),  \\
    -\nabla\cdot(\kappa_i\nabla u_i) &= f_i   && \quad \text{ in }\Omega_i\times (0,T) ,  \\
    \kappa_i\nabla u_i\cdot \bm n   &= \kappa_e \nabla u_e\cdot \bm n   && \quad \text{ on } \Gamma\times (0,T),   \\
  C_M \frac{\partial [u]}{\partial t} + f_\Gamma ([u]) & =  -\kappa_i \nabla u_i \cdot \bm n  && \quad  \text{ on }  \Gamma \times (0,T),  \\
    \nabla u_e \cdot \bm n_e  & = 0 &&  \quad \text{ on } \Gamma_e \times (0,T),   \\ 
    [u](0) & = \hat u^0 && \quad \text{ on } \Gamma.  
  \end{alignedat}
\end{equation}
Here,  $\bm n  = \bm n_i $ and $\bm n_e$ are the unit normal vectors to $\partial \Omega_i $ and $\partial \Omega_e$ respectively. The jump of $u$ is defined as $[u] = u_{\vert_{ \Omega_i}} - u_{\vert_{\Omega_e}}$. We assume that stimulation currents $f_e \in L^2(0,T;L^2(\Omega_e))$,  $f_i \in L^2(0,T;L^2(\Omega_i))$ are given.  We also assume that  $\hat u^0 \in H^{1/2}(\Gamma)$ and that the conductivities $\kappa_i \in L^\infty(\Omega_i, \mathbb{R}^{d,d})$ and $\kappa_e \in L^\infty(\Omega_e, \mathbb{R}^{d,d})$ are uniformly symmetric positive definite matrices in $\Omega_i $ and $\Omega_e$ respectively. The constant $C_M > 0$ represents the capacitance of the membrane.

Note that the compatibility condition $\int_{\Omega} f = 0$ where $f = f_i$ in $\Omega_i$ and $f = f_e$ in $\Omega_e$ is required due to the Neumann condition on $\Gamma_e\times (0, T)$. This boundary condition is natural in the absence of grounding of the potential $u_e$. 
We further assume that $f_{\Gamma} \in W^{1,\infty}(\mathbb R)$ is given and that $f_{\Gamma}([u])$ models the ionic current through the membrane $\Gamma$. Typically, model \eqref{eq:EMI} is coupled with a system of ordinary differential equations modeling ionic channels. The coupling is through the membrane where  $f_{\Gamma}$ may depend on the gating variables for ion channels. For simplicity, we do not consider these models in the analysis.

\subsection{Weak formulation} In this section, we introduce the weak formulation for the model in \eqref{eq:EMI} and we follow the exposition in~\cite{fokoue2023numerical, veneroni2006reaction}. Define the broken space  \begin{equation}
    W = \left \{ u \in  L^2(\Omega): \,\,\,  u \vert_{\Omega_i} \in H^1(\Omega_i), \,\, u \vert_{\Omega_e} \in H^1(\Omega_e), \,\,\, \int_{\Omega_e} u  = 0 \right \}.  \label{eq:def_W}
\end{equation}
This space is equipped with the norm 
\begin{equation}
    \|u\|_{W}^2 = \|u\|^2_{H^{1}(\Omega_i)} + \|u\|^2_{H^1(\Omega_e)}, \quad \forall u \in W. 
\end{equation}
Testing \eqref{eq:EMI} with $v \in W$ and integrating by parts, we obtain the following weak formulation. Find $u \in L^2(0,T;W)$ with $[u] \in L^2(0,T; H^{1/2}(\Gamma)) \cap H^1(0,T;L^2(\Gamma))$ such that the following holds for all $v \in W$ for a.e. in time: 
\begin{equation}
\label{eq:EMI_weak} 
\int_{\Gamma}\left( C_M \frac{\partial [u]}{\partial t} + f_{\Gamma}([u]) \right)[v] + \int_{\Omega_e} \kappa_e \nabla u_e \cdot \nabla v + \int_{\Omega_i} \kappa_i \nabla u_i \cdot \nabla v   = \int_{\Omega} f v.    
\end{equation}
Using the compatibility condition $\int_{\Omega} f = 0$ and testing \eqref{eq:EMI_weak} with $\tilde{v} = (v - 1/|\Omega_e| \int_{\Omega_e} v)$, it is easy to see that \eqref{eq:EMI_weak} holds for any $ v \in H^1(\Omega_i \cup \Omega_e)$:
\begin{equation}
H^1(\Omega_i \cup \Omega_e) = \left \{ u \in  L^2(\Omega): \,\,\,  u \vert_{\Omega_i} \in H^1(\Omega_i),  \,\, u \vert_{\Omega_e} \in H^1(\Omega_e)  \right \}. \label{eq:broken_H1}
\end{equation}
Hereinafter, we assume that the above problem is well-posed. We refer to~\cite{amar2005existence, veneroni2006reaction, franzone2002degenerate} for detailed expositions on the analysis; for example, see~\cite[Theorem 1]{franzone2002degenerate} for a precise statement on the existence and uniqueness of solutions. 

\section{Discontinuous Galerkin method} We introduce the semi-discrete DG formulation for \eqref{eq:EMI} in this section. 
\subsection{Notation and preliminaries} 
We consider a family of quasi-uniform conforming affine simplicial meshes $\{\mathcal{T}_h\}_{h >0}$  of $\Omega$. Define $H^1(\mathcal{T}_h)$ as the broken $H^1$ space corresponding to the mesh $\mathcal{T}_h$: 
\[H^1(\mathcal{T}_h) = \{ u \in L^2(\Omega): \;\; u_{\vert_{K}}  \in H^1(K), \;\; \forall K \in \mathcal{T}_h \}. \] 

For each element $K \in \mathcal{T}_h$, we denote by $h_K$ the diameter of $K$, and we denote the characteristic mesh size by $h = \max_{K \in \mathcal{T}_h} h_K$. We further assume that the triangulation is conforming to $\Gamma$ in the sense
that for all $K\in\mathcal{T}_h$ the intersection $\Gamma\cap\overline{K}$ is either a vertex or
one facet of $K$. We denote by $\mathcal{T}_{h,i}$ (resp. $\mathcal{T}_{h,e}$) the collection of elements intersecting $\Omega_i$ (resp. $\Omega_e$). By the assumptions on the mesh, we remark that $\mathcal{T}_{h,i} \cap \mathcal{T}_{h,e} = \emptyset$.

The collection of all
interior facets of $\mesh$ shall be denoted by
 $\mathcal{F}_h$
and will be partitioned into $\mathcal{F}_{i, h}$ containing facets in the interior of $\Omega_i$, $\mathcal{F}_{e, h}$ containing facets in the interior of $\Omega_e$, 
and $\mathcal{F}_{\Gamma, h}$ containing facets located on the interface $\Gamma$. 
For $F \in \mathcal{F}_{\Gamma,h}$, we set the normal  $\bm n _F = \bm n_i$ the outward normal to $\partial \Omega_i$. 
Further, given a
facet $F\in\mathcal{F}_h$ with a normal $\bm n_F$ pointing from $K_F^1$ to $K_F^2$, we define scalar jump and average operators for $u \in H^1(\mathcal{T}_h)$: 
\[
[u] = u \vert_{K_F^1} - u \vert_{K_F^2}, \quad \{u\} = \frac{1}{2}(u\vert_{K_F^1} +u \vert_{K_F^2}), \quad \forall F = \partial K_F^1 \cap \partial K_F^2. 
\]
If $F \subset \Gamma_e$,  the jump and the average of a function $u$ are taken as the single-valued trace of $u$. Similar definitions are used for vector-valued functions.

Given an integer $k\geq 1$, let $\mathbb P_k(K)$ be the space of polynomials of degrees up to $k$ on $K$
and define the broken polynomial space $$V^k_h = \left\{v\in L^2(\Omega):  \,\,\,  v_{\vert_K}\in \mathbb P_k(K), \quad \forall K \in \mathcal{T}_h \right\}. $$
Further, we denote by $V^k_{h,i}$ (resp. $V^k_{h,e}$) the space of broken polynomials of degree $k$ defined over $\mathcal{T}_{h,i}$ (resp. $\mathcal{T}_{h,e}$).  
 
We now define the form $a_h(\cdot,\cdot)$ corresponding to an interior penalty DG discretization, see the textbooks ~\cite{riviere2008discontinuous, di2011mathematical} for details on such methods.  For all $u,v \in V_h^k$, 
\begin{equation}\label{eq:poisson_DGk}
\begin{aligned}
  a_h(u, v) &= \sum_{K\in\mathcal{T}_h}\int_{K}\kappa\nabla u\cdot \nabla v - \sum_{F\in\mathcal{F}_{h} \backslash  \mathcal{F}_{\Gamma,h}} \int_{F} \{\kappa \nabla u\}\cdot \bm n_F [ v ] \\
  &
  +\epsilon\sum_{F\in\mathcal{F}_{h} \backslash  \mathcal{F}_{\Gamma,h}} \int_{F} \{\kappa \nabla v\}\cdot \bm n_F [u]
  +  \sum_{F\in\mathcal{F}_{h} \backslash  \mathcal{F}_{\Gamma,h}} \int_F 
  \frac{\gamma}{h_F}[u][v]
  \,\mathrm{d}s.
\end{aligned}
\end{equation}
Here, $\kappa = \kappa_i $ in $\Omega_i$, $\kappa = \kappa_e$ in $\Omega_e$, and $h_F= \lvert F \rvert^{1/(d-1)}$.  The parameter $\epsilon \in \{-1,0,1\}$ yields symmetric, incomplete, and non-symmetric discretizations respectively. Moreover, $\gamma=\gamma(k, d , \kappa)$ is a stabilization parameter to be chosen large enough (if $\epsilon \in \{-1,0\}$) so that the form $a_h$ is coercive: 
\begin{equation} \label{eq:coercivity}
    a_h(u_h, u_h) \geq C_{\mathrm{coerc}} |u_h|^2_{\DG},  \quad \forall u_h \in V_h^k, 
\end{equation}
where we define the DG semi-norm: 
\begin{equation}\label{eq:dgseminorm}
 |u|_{\DG}^2=  \sum_{K \in \mathcal{T}_h } \|\nabla u \|^2_{L^2(K)} + \sum_{F\in\mathcal{F}_h \backslash \mathcal{F}_{\Gamma, h}} \frac{\gamma}{h_F}\|[u]\|^2_{L^2(F)}, \quad \forall u \in H^1(\mesh).  
\end{equation}
The proof of the coercivity property \eqref{eq:coercivity} is immediate for $\epsilon = 1$ and follows from standard arguments, see for e.g. \cite[Chapter 2]{riviere2008discontinuous}, for $\epsilon \in \{-1,0\}$.  It is also standard to show that the form $a_h$ is bounded with respect to the $|\cdot|_{\DG}$ semi-norm:
\begin{equation} \label{eq:continuity}
a_h(u_h,w_h) \leq C_{\mathrm{cont}} |u_h|_{\DG} |w_h|_{\DG}, \quad \forall u_h , w_h \in V_h^k. 
\end{equation}
We will make use of the following discrete subspaces  $\tilde{V}_{h0}^k \subset \tilde{V}_h^k \subset V_h^k$:
\begin{align}\label{eq:discretesubspaces} 
 \tilde{V}_h^k &= \left \{v_h \in V_h^k:  \,\,\,  [v_h] = 0, \quad \forall F \in \mathcal{F}_{\Gamma,h} \right \} ,\\ 
  \tilde{V}_{h0}^k & =  \left \{ v_h \in \tilde V_h^k:  \quad \int_{\Omega_e} v_h = 0  \right\}. 
\end{align}
\begin{lemma}[Poincar\'e's inequality over $\tilde V_{h0}^k$] 
We have that 
\begin{equation}\label{eq:Poincare}
\|u_h\|_{L^2(\Omega)} \leq C_P |u_h|_{\DG}, \quad \forall u_h \in \tilde V_{h0}^k. 
\end{equation}
\end{lemma}
\begin{proof}
This follows from  Poincar\'e's inequality~\cite[Theorem 5.1, example 4.3]{brenner2003poincare}: 
\begin{equation}
\|u_h\|^2_{L^2(\Omega)} \leq C_P \left(|u_h|_{\DG}^2 + \sum_{F \in \mathcal{F}_{\Gamma, h}} \frac{\gamma}{h_F} \|[u_h]\|^2_{L^2(F)} + |\Psi(u_h)|^2 \right), \quad \forall u_h \in V_h^k, 
 \end{equation}
where $\Psi(u) = \frac{1}{|\Omega_e|}\int_{\Omega_e} u$. Using the definition of $\tilde V_{h0}^k$ yields the result. 
\end{proof}
We also define a broken polynomial space on $\Gamma$: 
\begin{align}
\hat V_h^k = \{ \hat v_h \in L^2(\Gamma): \quad \hat v_h \in \mathbb{P}^k(F), \quad \forall F \in \mathcal{F}_{\Gamma,h} \},   
\end{align}
along with the $L^{2}$-projection $\hat \pi_h: L^2(\Gamma) \rightarrow \hat V_h^k$ defined as follows. For any $\hat v \in L^2(\Gamma)$,
\begin{equation}
\int_{\Gamma} \hat \pi_h \hat v  \,\hat  \varphi_h = \int_{\Gamma} \hat v \, \hat \varphi_h , \quad \forall \hat\varphi_h \in \hat V_h^k.
\end{equation}
Throughout the paper, we use the standard notation $A \lesssim B$ if  $A \leq C B$ for some positive constant $C$ independent of the mesh parameters and the time step.

\subsection{Semi-discrete formulation} \label{sec:semi_discrete_posed} The semi-discrete DG approximation of the EMI problem in \eqref{eq:EMI} reads as follows. Find $u_h(t) \in V_h^k $  with $\int_{\Omega_e} u_h(t) = 0$ such that for a.e. in $(0,T)$, 
\begin{subequations}\label{eq:semi_discrete}
\begin{align}
\int_{\Gamma} C_M \frac{\partial [u_h]}{ \partial t} [v_h]  + \int_{\Gamma} f_\Gamma ([u_h]) [v_h] + a_h( u_h, v_h) & = \int_{\Omega} f v_h, \quad \forall v_h \in V_h^k, \\ 
[u_h](0) & = \hat \pi_h \hat u^0.
\end{align}
\end{subequations}

\textbf{Main goal and outline}. The focus of this section is showing well-posedness of the above formulation. This is attained by first transforming the problem to one posed on the membrane $\Gamma$, following ideas from~\cite{fokoue2023numerical}. To do so,  we introduce two lifting operators, $L_h, S_h: \hat V_h^k \rightarrow V_h^k$. The first is a linear lift denoted by $L_h$,  see Lemma \ref{lemma:lift_L}, and is naturally constructed to satisfy $[L_h g] = g$ on $\Gamma$. The second operator $S_h$ is constructed in Lemma \ref{lemma:S_h} to satisfy \eqref{eq:semi_discrete} for test functions in $\tilde V_h^k$. These two operators allow us to write a problem on $\Gamma$, see \eqref{eq:problem_Gamma_time}, which is equivalent to \eqref{eq:semi_discrete} as shown in \Cref{lemma:equivalance}. Uniqueness and existence of solutions are then established for the problem on $\Gamma$ in \Cref{thm:exist_unique} using the Cauchy-Lipschitz Theorem.

We begin with defining a linear lifting operator.  
\begin{lemma}\label{lemma:lift_L}
 There is a linear operator $L_h:   \hat V_h^k \rightarrow V_h^k$ satisfying the following properties.  
\begin{multline}
\label{eq:property_Lh}
\|L_h g \|_{L^2(K)} \lesssim h_{K}^{1/2} \|g \|_{L^2(F)}, \,\,\,  \mathrm{for} \,\,\,  F \subset \partial K,  \\ 
[L_h  g ] = g , \,\,\, \mathrm{ on } \,\,\,  F \in \mathcal{F}_{\Gamma,h}, \quad \mathrm{and} \quad  \int_{\Omega_e} L_h g  = 0 .   
\end{multline}
\end{lemma}
\begin{proof} 
The construction of $L_h$ is natural. First, for $K \in \mathcal{T}_h, $ with $|K \cap \Omega_i| = \emptyset $, we set $L_h g = 0 $. 
Since the mesh is assumed to conform to the interface, any $F \in \Gamma_h$ is such that $F =  \partial  K_F \cap \partial \Omega_i$ for $K_F \subset \Omega_i$. 
On $K_F$, we naturally extend $g$ to the interior of $K_F$. Precisely, let $A$ denote the invertible affine mapping between the reference simplex $\tilde K_F$ and $K_F$. Then for $\tilde x \in \tilde K_F$, define $\tilde L_h g (\tilde x)= g(A( \pi \tilde x))$ where $\pi$ is the orthogonal projection onto  $\tilde F$ (the reference facet). Note that such a projection is affine linear and is always possible on the reference element.  For $x \in K_F$, we then set $L_h g (x) = \tilde L_h g (A^{-1} x). $  It is clear that $L_h g \in V_h^k$ since it is a composition of affine linear functions and $g \in \mathbb P^k(F)$. 

The second property holds since for any $F \in \mathcal{F}_{\Gamma,h}$, $F \in \partial K_F^1 \cap \partial K_F^2$  with $|K_F^1 \cap \Omega_i| \neq \emptyset$,  we have  that  $ L_h g \vert_{K_F^1} = g $ and $L_h g  \vert_{K_F^2} = 0$ on $\Gamma$.  The third property holds trivially since $L_h g = 0$ on $\Omega_e$. To show the first bound, we use the shape regularity of the mesh and estimate
\begin{multline}   
\|L_h g \|_{L^2(K_F) } \lesssim |K_F|^{1/2} \| \tilde L_h g \|_{L^2(\tilde K_F)} \lesssim |K_F|^{1/2}  \| \tilde L_h g \|_{L^2(\tilde F)} \\  \lesssim |K_F|^{1/2} |F|^{-1/2}\|  L_h g \|_{L^2( F)} \lesssim h_{K_F}^{1/2}\|  L_h g \|_{L^2( F)}.   
\end{multline}
 In the above, we used norm equivalence since $\|\cdot \|_{L^2(\tilde F)}$ defines a norm over the space of polynomials that  are constant along the normal of $\tilde F$. 
\end{proof}
We now construct the operator $S_h$. First, for a given $\hat v \in \hat V_h^k$,  we define  $\tilde u_{h, \hat v}  \in \tilde V_{h0}^k$ solving
\begin{equation}
a_h(\tilde u_{h, \hat v} , w_h )  = \int_{\Omega} f w_h - a_h(L_h \hat v , w_h), \quad \forall w_h \in  \tilde V_h^k. \label{eq:problem_tilde}  
\end{equation}
\begin{lemma}
For $\hat v \in \hat{V}_h^k$, there exists a unique solution $\tilde u_{h,\hat v} \in \tilde V_{h0}^k$ to \eqref{eq:problem_tilde}. 
\end{lemma}
\begin{proof}
From the coercivity property \eqref{eq:coercivity} over $\tilde V_{h0}^k$, and  Poincar\'e's inequality \eqref{eq:Poincare}, an application of the Lax-Milgram theorem shows that there exists a unique solution $\tilde u_{h, \hat v} \in \tilde V_{h0}^k $ solving \eqref{eq:problem_tilde} for all $w_h \in \tilde V_{h0}^k$. To see that $\tilde u_{h, \hat v} \in \tilde V_{h0}^k$ solves \eqref{eq:problem_tilde} for any $w_h \in \tilde V_h^k$, test \eqref{eq:problem_tilde} with $(w_h - \frac{1}{|\Omega_e|} \int_{\Omega_e} w_h) \in \tilde V_{h0}^k$, use the fact that $a_h(v_h, C) = 0$ for any $v_h $ and any constant $C$, and use the compatibility condition that $\int_\Omega f = 0$.
 \end{proof} 
Thus, we can construct the operator $S_h: \hat{V}_h^k \rightarrow V_h^k$ satisfying the following properties. 
\begin{lemma} \label{lemma:S_h}
The operator $ S_h: \hat{V}_h^k \rightarrow V_h^k$ defined by $S_h(\hat v) = \tilde u_{h,\hat v} + L_h(\hat v)$  solves
\begin{equation}
a_h(S_h(\hat v ), w_h) = \int_{\Omega} f w_h, \quad \forall w_h \in \tilde V_h^k.   \label{eq:problem_Gamma}
\end{equation} 
Further, we have that 
\begin{equation}
[S_h (\hat v)]= \hat v  \,\, \mathrm{ on } \,\, \mathcal{F}_{\Gamma,h}, \quad \mathrm{and}  \quad \int_{\Omega_e} S_h \hat v = 0.  \label{eq:properties_Sh}
\end{equation}
In addition, for $g_1, g_2 \in \hat V_h^k$, 
we have that 
\begin{equation}
|S_h(g_1) - S_h(g_2)|_{\DG} \lesssim  |L_h(g_1 - g_2) |_{\DG}.  \label{eq:Lipschitz_Sh}
\end{equation}
\end{lemma}
\begin{proof}
The identity \eqref{eq:problem_Gamma} follows from the definition of $S_h$ and \eqref{eq:problem_tilde}.   The  properties in \eqref{eq:properties_Sh} hold since  $\tilde u_{h,\hat v} \in \tilde{V}_{h,0}
^k$, $[S_h(\hat v)]  = [\tilde u_{h,\hat v}] +[L_h \hat v] = [L_h \hat v] = \hat v$ and $\int_{\Omega_e} L_h (\hat v) = 0$. For the last bound, note that for $g_1, g_2$ we have that for all  $w_h \in \tilde V_h^k $
\begin{equation}
a_{h} (\tilde u _{h,g_1} - \tilde u _{h,g_2}, w_h) = a_h(L_h (g_2) - L_h(g_1), w_h) = a_h(L_h(g_2-g_1), w_h). 
\end{equation}
In the above, we use the linearity of $L_h$. Testing with $w_h =\tilde u _{h,g_1} - \tilde u _{h,g_2} \in \tilde V_h^k$ and using the coercivity property \eqref{eq:coercivity} and continuity of $a_h$ \eqref{eq:continuity}, we obtain 
\begin{align}
 | \tilde u _{h,g_1} - \tilde u _{h,g_2} |^2_{\DG} & \lesssim |L_h(g_1-g_2)|_{\DG}  |\tilde u _{h,g_1} - \tilde u _{h,g_2}|_{\DG}. 
\end{align} 
The result then follows from the triangle inequality.
\end{proof}

With $S_h$ and $L_h$, we consider the following problem on $\Gamma$. Find $\hat u_h(t) \in \hat V_h^k$ satisfying for a.e. in $(0,T)$ 
 \begin{subequations}
\label{eq:problem_Gamma_time}
\begin{align}
    \int_{\Gamma} C_M \frac{\partial \hat u_h}{\partial t} \hat \varphi + \int_{\Gamma} f_{\Gamma}(\hat u_h) \hat \varphi + a_h(S_h \hat u_h, L_h \hat \varphi ) & = \int_{\Omega} f L_h \hat \varphi,   \quad \forall \hat \varphi \in \hat{V}_h^k,   \\ 
    \hat u_h(0) = \hat \pi_h u^0.
    \end{align}
\end{subequations}
\begin{lemma}[Equivalence]\label{lemma:equivalance}Problem \eqref{eq:problem_Gamma_time} has a unique solution if and only if problem \eqref{eq:semi_discrete} has a unique solution. 
 \end{lemma}
\begin{proof}
(\textit{Existence of solutions}) Assume that \eqref{eq:semi_discrete} has a solution $u_h(t)$. Then, we show that  $\hat u_h = [u_h] \in \hat V_h^k$ solves \eqref{eq:problem_Gamma_time}. First, observe that $u_h = S_h([u_h])$. Indeed, by testing \eqref{eq:semi_discrete} with $w_h \in \tilde V_{h}^k$, it is easy to see that $u_h - L_h([u_h]) \in \tilde V_{h0}^k$ solves \eqref{eq:problem_tilde} for $\hat v = [u_h]$. By uniqueness, it follows that $u_h - L_h([u_h]) = \tilde u_{h,[u_h]}$ and therefore $u_h = \tilde u_{h,[u_h]} + L_h([u_h]) = S_h ([u_h])$. Thus, for a given $\hat \varphi \in \hat V_h^k$, test \eqref{eq:semi_discrete} with $L_h \hat \varphi \in V_h^k$ to see that $[u_h]$ solves \eqref{eq:problem_Gamma_time}.

Conversely, assume that $\eqref{eq:problem_Gamma_time}$  has a  solution $\hat u_h(t)$. Then, we show that $u_h = S_h \hat u_h$ solves \eqref{eq:semi_discrete}. First recall  that $\int_{\Omega_e} S_h \hat u_h = 0$. For a given $v_h \in V_h^k$, test \eqref{eq:problem_Gamma_time} with $[v_h]{}_{\vert_{\Gamma}} \in \hat V_h^k$. Then, with also using that $[u_h] = [S_h(\hat u_h)]  = \hat u_h$ on $\Gamma$, we have that 
\begin{equation}
\int_{\Omega} C_M \frac{\partial [u_h]}{\partial t}[v_h] + \int_{\Gamma } f_{\Gamma}([u_h]) [v_h] + a_h(u_h, L_h ([v_h]))  = \int_{\Omega} f L_h ([v_h]).  
\end{equation}
Testing \eqref{eq:problem_Gamma} with $w_h = v_h - L_h ([v_h])$ ($w_h \in \tilde V_h^k,$ since $[w_h] = 0$ on $\Gamma$) and adding it to the above  equation shows that $u_h = S_h \hat u_h$ solves \eqref{eq:semi_discrete}. 

(\textit{Uniqueness}) Assume that \eqref{eq:problem_Gamma_time} has a unique solution, and suppose that there are two different solutions $u_h^1, u_h^2 \in V_h^k$ to \eqref{eq:semi_discrete}. There are two cases to consider: (i) $[u_h^1 - u_h^2] = 0 $ on $\Gamma$, and (ii) $[u_h^1 - u_h^2] \ne 0 $ on $\Gamma$. If $[u_h^1 - u_h^2] = 0 $ on $\Gamma$, then by testing \eqref{eq:semi_discrete} by $w_h \in \tilde V_h^k$ for each solution and subtracting, we obtain 
 \begin{equation*}
     a_h(u_h^1 - u_h^2, w_h) = 0, \quad \forall w_h \in \tilde{V}_h^k. 
 \end{equation*}
 From the coercivity property \eqref{eq:coercivity}  and Poincar\'e's inequality \eqref{eq:Poincare} (since $u_h^1 - u_h^2 \in \tilde{V}_{h0}^{k}$), it follows that $\|u_h^1 - u_h^2\|_{L^2(\Omega)} = 0$ showing that the solution must be unique. If $[u_h^1] \neq [u_h^2]$, then, as shown above, $[u_h^1], [u_h^2] \in \hat{V}_h^k$ solve \eqref{eq:problem_Gamma_time} contradicting the uniqueness of \eqref{eq:problem_Gamma_time}. 

 Assume now that \eqref{eq:semi_discrete} has a unique solution, and suppose that there are two solutions $\hat u_h^1, \hat u_h^2 \in  V_h^k$ to \eqref{eq:problem_Gamma_time}. Then, as above, we can show that $S_h(\hat u_h^1) $ and $S_h (\hat u_h^2) $ solve \eqref{eq:semi_discrete}. By uniqueness of \eqref{eq:semi_discrete}, $S_h (\hat u_h^1) = S_h (\hat u_h^2)$. By Lemma \ref{lemma:S_h}, $\hat u_h^1 = [S_h(\hat u_h^1)] = [S_h(\hat u_h^2)] = \hat u_h^2$. Hence,  we conclude that the solutions to \eqref{eq:problem_Gamma_time} must be unique.
\end{proof}


\begin{theorem}[Existence and uniqueness of solutions]\label{thm:exist_unique}There exists a unique solution $u_h(t)$ to \eqref{eq:semi_discrete} with $[u_h](t) \in C^1([0,T];\hat V_h^k)$ and $[u_h](0) = \hat \pi_h \hat u^0. $
 \end{theorem}
 \begin{proof}
From \Cref{lemma:equivalance}, it suffices to show that problem \eqref{eq:problem_Gamma_time} has a unique solution. We look for a solution of the form 
\begin{equation}
\hat u_{h, \alpha} (t,\bm x) = \sum_{F \in \mathcal{F}_{\Gamma,h} } \sum_{j=1}^{k+1} \tilde{\alpha}_i^F (t) \tilde \phi_{i}^F(\bm x)  = \sum_{i=1}^{N} \alpha_i(t) \phi_i (\bm x),    \label{eq:looking_for_sol}
\end{equation}
where $\{\tilde \phi_i^F\}$ are local basis functions on each face $F \in \mathcal{F}_{\Gamma,h}$ extended by $0$ to the remaining faces. The second sum is a renaming of the basis functions and coefficients to simplify notation. The vector $ \alpha (t) = (\alpha_1(t), \ldots, \alpha_{N}(t))$ solves the following ODE: 
\begin{equation}
M_{\Gamma} \frac{d}{dt} \alpha(t) = F( \alpha(t)), \quad  t \in (0,T) ,  
\end{equation}
with $\bm \alpha(0)$ being the coefficients of the discrete function $\hat \pi_h \hat u^0.$  Here, we set 
\begin{alignat}{2}
   (M_{\Gamma})_{i,j} & = C_M \int_{\Gamma } \phi_i \phi_j, &&  \quad 1 \leq i,j \leq N, \\ 
  ( F( \alpha (t) ))_{i} & = \int_{\Omega} f L_h \phi_i - \int_{\Gamma} f_{\Gamma} (\hat u_{h,\alpha}) \phi_i - a_h  ( S_h \hat u_{h,\alpha} , L_h \phi_i  ), && \quad 1 \leq i \leq N. 
\end{alignat}
Since the mass matrix $M_{\Gamma}$ is invertible, the existence and uniqueness of solutions for the above ODE follows from an application of the Cauchy-Lipschitz theorem, see~\cite[Theorem 3.8-1]{ciarlet2013linear}, after checking that $F$ is Lipschitz with a constant independent of $t$. To this end, note that for any $ \alpha(t), \beta(t) \in \mathbb{R}^N$ and $\phi_i$,  
\begin{align}
|a_h(S_h \hat u_{h,\alpha} , L_h \phi_i ) - a_h(S_h \hat u_{h,\beta} , L_h \phi_i ) | & = |a_h(S_h(\hat u_{h,\alpha}) - S_h(\hat u_{h,\beta}), L_h \phi_i)| \\ 
& \lesssim |S_h(\hat u_{h,\alpha})  - S_h(\hat u_{h,\beta} )|_{\DG} |L_h \phi_i|_{\DG} \nonumber \\ 
& \lesssim \gamma_h \|\hat u_{h,\alpha} - \hat u_{h,\beta}\|_{L^2(\Gamma)} \|\phi_i\|_{L^2(\Gamma)} 
.
\nonumber
\end{align}
In the above, we used \eqref{eq:Lipschitz_Sh} in \Cref{lemma:S_h}, inverse estimates, and \eqref{eq:property_Lh}
\footnote{ 
Denote by $\mathcal{T}_{\Gamma}$ the collection of elements intersecting the interface $\Gamma$ and $\Omega_i$. For any $g \in \hat V_h^k$, we use the definition of $L_h g$ along with an inverse estimate and \eqref{eq:property_Lh} to estimate 
\begin{equation*}
|L_h g|_{\DG}^2 = \sum_{K \in\mathcal{T}_{\Gamma}} \|\nabla L_h g\|^2_{L^2(K)} \lesssim \sum_{K \in \mathcal{T}_\Gamma} h_{K}^{-2} \|L_h g\|_{L^2(K)}^2  \lesssim  \sum_{F \in \mathcal{F}_\Gamma}  h_{K}^{-1} \| g\|_{L^2(F)}^2 \lesssim (\max_{K \in \mathcal{T}_{\Gamma}} h_{K}^{-1}) \|g\|_{L^2(\Gamma)}^2. 
\end{equation*}
}. Thus, the constant $\gamma_h$ depends on $h$. Further, since $f_{\Gamma}$ is Lipschitz it follows that 
\begin{equation}
\int_{\Gamma} ( f_{\Gamma} (\hat u_{h,\alpha}) - f_{\Gamma}(\hat u_{h,\beta}) )  \phi_i \lesssim  \|\hat u_{h,\alpha} - \hat u_{h,\beta} \|_{L^2(\Gamma) } \| \phi_i\|_{L^2(\Gamma)}
.
\end{equation}
From here, one can see that 
\begin{equation}
   \| F(\alpha(t) ) - F(\beta(t))\|_{\ell^2} \leq C \|\alpha(t) - \beta(t)\|_{\ell^2}, \quad \forall t \in [0,T] , \,\,\, \forall \alpha,\beta \in \mathbb{R}^N. 
\end{equation}
The constant $C$ depends on $h$ but is independent of $t$. Thus, an application of~\cite[Theorem 3.8-1]{ciarlet2013linear} shows that a unique solution $\alpha(t) \in C^1([0,T], \mathbb{R}^N)$ exists, implying that a unique solution to  \eqref{eq:semi_discrete} with $[u_h](t) \in C^1([0,T];\hat V_h^k)$ and $[u_h](0) = \hat \pi_h \hat u^0$ exists.
\end{proof}

\section{Stability} \label{sec:stability} We show that the semi-discrete formulation is stable in the DG semi norm \eqref{eq:dgseminorm}. The main result is given in \Cref{thm:stability} which is based on establishing a discrete Poincar\'e inequality over $V_h^k$ in \Cref{lemma:poincare_discrete}. In this Lemma, averaging operators and a Poincar\'e inequality over $W$ (\Cref{lemma:poincare_W}) are utilized. 
\begin{lemma}[Poincar\'e's inequality over $W$]  \label{lemma:poincare_W} Recall the definition of $W$ in \eqref{eq:def_W}. For any $w \in W$, we have that 
\begin{equation}
\|w\|_{L^2(\Omega)} \lesssim \|\nabla w\|_{L^2(\Omega_i)} + \|\nabla w\|_{L^2(\Omega_e)} + \|[w]\|_{L^2(\Gamma)}.
\end{equation}
\end{lemma}
\begin{proof}
We use the following Poincar\'e inequality, see e.g.~\cite[Lemma B.63]{ern2004theory} 
\begin{equation}
\|w\|_{L^2(\Omega_i) } \lesssim  \|\nabla w\|_{H^1(\Omega_i)} + \|w _{\vert \Omega_i} \|_{L^2(\Gamma)}.
\end{equation}
For the last term above, we use triangle and trace inequalities to obtain
\begin{equation}
\|w\|_{L^2(\Omega_i) } \lesssim  \|\nabla w\|_{H^1(\Omega_i)} + \|[w]\|_{L^2(\Gamma)} +\|w_{\vert \Omega_e} \|_{H^1(\Omega_e)}. \label{eq:poin_W_0}
\end{equation}
On $\Omega_e$, since $ \int_{\Omega_e} w = 0$, we have that 
\begin{equation}
\|w\|_{L^2(\Omega_e)} \lesssim \|\nabla w\|_{L^2(\Omega_e)}.  \label{eq:poin_W_1}
\end{equation}
Using the above in \eqref{eq:poin_W_0} and adding the resulting bound to \eqref{eq:poin_W_1} yields the result. 
\end{proof}
\begin{lemma}[Poincar\'e's inequality over $V_h^k$]\label{lemma:poincare_discrete} For any $u_h \in V_h^k \,\, \mathrm{with } \,\, \int_{\Omega_e} u_h = 0 $, we have that 
\begin{equation} \label{eq:poincare_discrete_global}
\|u_h\|^2_{L^2(\Omega)} \leq C_p \left( |u_h|^2_{\DG} +  \|[u_h]\|^2_{L^2(\Gamma)} \right). 
\end{equation}
\end{lemma}
\begin{proof}
The proof utilizes an averaging operator similar to the one constructed in ~\cite{doi:10.1137/S0036142902405217} and ~\cite[Section 5]{cangiani2018adaptive}. In particular, we construct an operator $E_h: V_h^k \rightarrow H^1(\Omega_i) \cup H^1(\Omega_e)$ as follows.
Let $E_h^i: V^k_{h,i} \rightarrow V^1_{h,i} \cap H^1(\Omega_i)$ be the local averaging map defined over $\mathcal{T}_{h,i}$ and let $E_h^e: V^k_{h,e} \rightarrow V^1_{h,e} \cap H^1(\Omega_e)$ be the local averaging  map defined over $\mathcal{T}_{h,e}$. We refer to ~\cite[Section 4]{ern2017finite} for a precise definition.  Then, we define the map $E$ as $E v_h \vert_K = E_h^{i} v_h \vert_K $  for $K \in \mathcal T_{h,i} $ and  $E  v_h \vert_K = E_h^{e} v_h \vert_K $  for $K \in \mathcal T_{h,e}$. Further, from ~\cite[Lemma 4.3]{ern2017finite}, we have that
\begin{align}
\sum_{K \in \mathcal{T}_h }\|\nabla (u_h - E u_h) \|_{L^2(K)}^2  &\lesssim \sum_{F \in \mathcal{F}_h \backslash \mathcal F_{\Gamma, h}} \frac{1}{h_F }\|[u_h]\|_{L^2(F)}^2, \label{eq:enriching_map} \\ 
\|u_h - E u_h \|_{L^2(\Omega)}^2  & \lesssim h  ^2 \sum_{F \in \mathcal{F}_h \backslash \mathcal F_{\Gamma, h} } \frac{1}{h_F} \|[u_h]\|_{L^2(F)}^2. \label{eq:enriching_map_1} 
\end{align}
For $E u_h$,  we let $\tilde{E} u_h= E u _h - \langle E u_h \rangle $ in $\Omega_e$ and $\tilde E u_h = Eu_h$ in $\Omega_i$ where $\langle \cdot \rangle$ denotes the average operator over $\Omega_e$. Using the triangle inequality and the fact that $\langle u_h \rangle = 0$, we obtain 
\begin{equation}
\|E u_h\|_{L^2(\Omega)} \leq \| \tilde E u_h \|_{L^2(\Omega_e)} + \|\langle E u_h - u_h \rangle \|_{L^2(\Omega_e) } + \|\tilde E u_h\|_{L^2(\Omega_i)}.  
 \end{equation} 
 With \Cref{lemma:poincare_W} and the triangle inequality, we estimate 
\begin{multline}
\| \tilde E u_h\|_{L^2(\Omega)}   \lesssim \|\nabla  (E u_h) \|_{L^2(\Omega_i)} + \|\nabla (Eu_h)\|_{L^2(\Omega_e)}    +\|[u_h]\|_{L^2(\Gamma)}+ \|[u_h - \tilde E u_h  ]\|_{L^2(\Gamma)}. 
\end{multline}
For the last term above,  we invoke a discrete trace estimate locally  over each $F\in \mathcal{F}_{\Gamma,h}$ and sum over the faces, to obtain that  \begin{multline}
\|[u_h - \tilde E u_h]\|_{L^2(\Gamma)} \lesssim h^{-1/2} \|u_h - \tilde E u_h\|_{L^2(\Omega)} \\ \lesssim  h^{-1/2} (\|u_h -  E u_h\|_{L^2(\Omega)} + \|\langle u_h -  Eu_h \rangle\|_{L^2(\Omega_e)} ), 
\end{multline}
where we again use that $\langle u_h \rangle = 0$. By noting that $\|\langle u - E u_h \rangle \|_{L^2(\Omega_e)} \leq \| u - E u_h \|_{L^2(\Omega_e)}$, \eqref{eq:enriching_map}--\eqref{eq:enriching_map_1} and the above bounds yield
\begin{align}
\|E u_h \|_{L^2(\Omega)} \lesssim |u_h|_{\DG} + \|[u_h]\|_{L^2(\Gamma)}.  
\end{align}
Estimate \eqref{eq:poincare_discrete_global} follows by an application of the triangle inequality and  \eqref{eq:enriching_map_1}.
\end{proof}

\begin{theorem}[Stability] \label{thm:stability} Let $u_h \in V_h^k $ with $\int_{\Omega_e} u_h = 0$ solve \eqref{eq:semi_discrete}. Then, there exists a constant $C_T$  depending on $T$ but not on $h$ such that for any $t \in [0,T]$, 
\begin{equation}
\|[u_h](t)\|_{L^2(\Gamma)}^2 + \int_0^T |u_h|^2_{\DG} \leq C_T (   \|\hat u^0\|_{L^2(\Gamma)}^2 + \int_{0}^T (\|f\|^2 _{L^2(\Omega)} +  \|f_{\Gamma}(0)\|_{L^2(\Gamma)}^2) ). 
\end{equation}

\end{theorem}
\begin{proof}
Testing \eqref{eq:semi_discrete} with $u_h$ and using \eqref{eq:coercivity}, we obtain that 
\begin{equation}\nonumber
\frac{C_M}{2} \frac{d}{dt} \|[u_h]\|_{L^2(\Gamma)}^2 + C_{\mathrm{coerc}} |u_h|^2_{\DG} \leq \|f\|_{L^2(\Omega)} \|u_h\|_{L^2(\Omega)} + \|f_{\Gamma}([u_h])\|_{L^2(\Gamma)} \|[u_h]\|_{L^2(\Gamma)}. 
 \end{equation}
 Since $f_{\Gamma}$ is Lipschitz, we estimate 
 \begin{align} \label{eq:bound_f_Gamma}
    \|f_{\Gamma}([u_h])\|_{L^2(\Gamma)}  & \leq  \|f_{\Gamma}([u_h]) - f_{\Gamma}(0)\|_{L^2(\Gamma)} + \|f_{\Gamma}(0)\|_{L^2(\Gamma)}
    \\ & \lesssim 
    \nonumber 
    \|[u_h]\|_{L^2(\Gamma)} + \|f_{\Gamma}(0)\|_{L^2(\Gamma)}. 
 \end{align}
 With the above bound, \Cref{lemma:poincare_discrete}, and appropriate applications of Young's inequality, we obtain that  
 \begin{align}
     \frac{C_M}{2} \frac{d}{dt} \|[u_h]\|_{L^2(\Gamma)}^2 + \frac{C_{\mathrm{coerc}}}{2} |u_h|^2_{\DG} \lesssim \|f\|^2 _{L^2(\Omega)} + \|[u_h]\|_{L^2(\Gamma)}^2 +  \|f_{\Gamma}(0)\|_{L^2(\Gamma)}^2. 
 \end{align}
 Integrating the above over $[0,t]$ and using a continuous Gronwall's inequality and the stability of the $L^2$ projection for the initial condition yields the result.  
 \end{proof}

\section{Error analysis} \label{sec:convergence} In this section, we prove a priori error estimates for the semi-discrete formulation. 
Since the domains $\Omega_i$ and $\Omega_e$ have corner singularities, we work with the spaces defined for $s \in (0,1/2)$
\begin{align*} 
V^{1+s}(\Omega)  &= \{u  \in H^{1+s}(\Omega_i \cup \Omega_e):  \nabla \cdot(\kappa_i \nabla u \vert_{\Omega_i}) \in L^2(\Omega_i), \,\, \nabla \cdot(\kappa_e \nabla u \vert_{\Omega_e}) \in L^2(\Omega_e)\}, \\ 
V^{1+s}(\mathcal{T}_h) &  =  V^{1+s}(\Omega) + V_h^k, 
\end{align*}
where $ H^{1+s}(\Omega_i \cup \Omega_e)$ and $H^{1+s}(\mathcal{T}_h)$  denote the  broken Sobolev spaces of degree $1+s$ over $\Omega_i \cup \Omega_e$ and over  $\mathcal{T}_h$ respectively. 
The space $V^{1+s}(\mathcal{T}_h)$ was considered for DG methods for elliptic interface problems in ~\cite{cai2011discontinuous}. 

\textbf{Main difficulties and outline}. Due to the low regularity of the solution, strong consistency of the DG approximation cannot be expected. In particular, the exact solution cannot be used as an argument in the discrete bilinear form \eqref{eq:poisson_DGk} since the normal trace of the flux $\kappa \nabla u$ on mesh facets is not well defined. This precludes an a priori analysis based on C\'{e}a's lemma as was done in ~\cite{fokoue2023numerical}. However, following ~\cite{ern2021quasi}, weak meaning can be given to the normal trace of the flux through duality and this will allow us to perform our analysis using a modification of Strang's Second Lemma. From ~\cite{ern2021quasi}, we utilize the face-to-element lifting to show a notion of weak consistency given in \Cref{lemma:weak_consistency}. The main result is given in \Cref{thm:convergence} which requires an intermediate lemma on the consistency error given in \Cref{lemma:bounding_W3}.

\textit{The face-to-element lifting}.
Let $K \in \mathcal{T}_h$ be a mesh element, $\mathcal{F}_K$ the set of all faces of $K$, and $F \in \mathcal{F}_K$ be a face of $K$. For $F \in \mathcal{F}_h$, $\mathcal{T}_F$ is the set containing the two elements sharing $F$. First, we note that if $u \in H^{1+s}(K)$, then $\nabla u \in (H^s(K))^d$. Since $2s < 1 < d$, the Sobolev embedding theorem implies that $(H^s(K))^d \hookrightarrow (L^\rho(K))^d$ for all $2 \le \rho \le 2d/(d-2s)$ ~\cite[Theorem 2.31]{Ern:booki}. In particular, we choose $\rho = 2d/(d-2s) > 2 $  such that $\nabla u \in (L^\rho(K))^d$, and since $\kappa \in (L^\infty(K))^{d\times d}$, we conclude that $\kappa \nabla u \in (L^\rho(K))^d$. This also shows that if $u \in V^{1+s}(\Omega)$, then $u_{\vert_{\Omega_i}} \in W^{1,\rho}(\Omega_i)$ and  $u _{\vert_{\Omega_e}} \in W^{1,\rho}(\Omega_e)$.
Now, consider the lifting operator given in ~\cite[Lemma 3.1]{ern2021quasi}: 
\begin{equation}
L^K_F: W^{\frac{1}{\rho}, \rho'}(F) \rightarrow W^{1, \rho'}(K), \quad \gamma_{\partial K} ( L_F^K(\phi) ) = \begin{cases} \phi & \mathrm{on} \,\,\, F, 
\\ 
0 & \mathrm{otherwise}. 
\end{cases}
\end{equation}
Here $1/\rho' + 1/\rho = 1$ and  $\gamma_{\partial K}$ denotes the Dirichlet trace operator onto $\partial K$.  The normal trace of a function $\bm{\tau} $ in $H^s(\mathcal{T}_h)^d$ with $\nabla \cdot \bm \tau \in L^2(\Omega)$ is then defined as a functional on $W^{1/\rho, \rho'}(F)$: 
\begin{equation}
\langle ( \bm \tau \cdot \bm n_K)_{\vert F}, \phi\rangle_{F} := \int_K  (\bm \tau \cdot \nabla L_F^K(\phi) + (\nabla \cdot \bm \tau) L_F^K(\phi)  ).   \label{eq:def_duality_pair}
\end{equation}
The above is well--defined since $L_F^K \in W^{1,\rho'} (K) \hookrightarrow L^2(K)$ and $\bm \tau \in L^{\rho}(K)^d$. 
Following~\cite{ern2021quasi} with the necessary adaptation to our setting, we define for all $u \in V^{1+s}(\mesh)$ and all $w_h \in V_h^k$
\begin{equation}
n_\sharp(u, w_h) = \sum_{ F \in \mathcal{F}_h \backslash \mathcal{F}_{\Gamma,h}}  \sum_{K \in \mathcal{T}_F} \frac12  \epsilon_{K,F}  \langle (\bm \sigma(u)_{\vert K} \cdot \bm n_{K} )_{\vert {F}} , [w_h] \rangle_{F},  \label{eq:def_eta_sharp}
 \end{equation}
where $\epsilon_{K,F} = \bm n_{F} \cdot (\bm n_{K})_{\vert {F}}$ accounts for the orientation and $\bm \sigma(u) =  - \kappa_i \nabla u_i $ in $\Omega_i$ and $\bm \sigma(u) = -  \kappa_e \nabla u_e$ in $\Omega_e$.  The definition of $\eta_{\sharp}(\cdot, \cdot): V^{1+s}(\mathcal{T}_h) \times V_h^k \rightarrow \mathbb R$ allows us to have a notion of weak consistency in the following sense. 
\begin{lemma}[Weak Consistency] \label{lemma:weak_consistency}Let $u \in L^2(0, T ; V^{1+s}(\Omega) \cap W)$ with $\partial_t u \in L^2(0,T; H^{1/2+s}(\Omega_i \cup \Omega_e))$ for $s \in (0,1/2)$  solve \eqref{eq:EMI_weak}. Then, for a.e. in time, 
\begin{equation}
\int_{\Gamma}\left( C_M \frac{\partial [u]}{\partial t} + f_{\Gamma}([u]) \right)[w_h]   + \sum_{K \in \mesh} \int_K  \kappa \nabla u \cdot \nabla w_h   +  \eta_\sharp(u,w_h)   = \int_{\Omega} f w_h,  \,\,\, \forall w_h \in V_h^k. 
\end{equation}
\end{lemma}
\begin{proof}
The proof utilizes some arguments from ~\cite[Lemma 3.3]{ern2021quasi}.  In particular, we make use of the mollification operators constructed in ~\cite{ern2016mollification}: 
\begin{equation}
\mathcal K_{\delta,i}^d: L^1(\Omega_i)^d \rightarrow C^{\infty}(\overline{\Omega_i})^d,  \quad \mathcal K_{\delta,i}^b : L^1(\Omega_i) \rightarrow C^{\infty}(\overline{\Omega_i}), 
\end{equation}
which satisfy the commutativity property for all $ \bm \tau \in L^1(\Omega_i)^d$ with $\nabla \cdot \bm \tau \in L^1(\Omega_i)$
\begin{equation}
    \nabla \cdot \mathcal K_{\delta,i}^d (\bm \tau) = \mathcal K_{\delta,i}^b(\nabla \cdot \bm \tau). \label{eq:commutative_prop_moll} 
\end{equation} 
We similarly define $\mathcal K_{\delta,e}^d$ and  $\mathcal K_{\delta,e}^b$. We denote by $\mathcal K_{\delta}^d$ (resp.  $\mathcal K_{\delta}^b$) the operator which evaluates to either $\mathcal K_{\delta,i}^d$ or $\mathcal K_{\delta,e}^d$ (resp. $\mathcal K_{\delta,i}^b$ or $\mathcal K_{\delta,e}^b$) depending on the domain of the function. We then define 
\begin{equation} \label{eq:def_eta_sharp_delta}
    n_{\sharp,\delta}(u, w_h) = \sum_{F \in \mathcal{F}_h \backslash \mathcal{F}_{\Gamma,h}}  \sum_{K \in \mathcal{T}_F} \frac12  \epsilon_{K,F}  \langle (\mathcal K_{\delta}^d(\bm \sigma(u))_{\vert K} \cdot \bm n_{K} )_{\vert {F}} , [w_h] \rangle_{F}. 
\end{equation}
Following the same arguments in ~\cite[Lemma 3.3]{ern2021quasi}, we readily obtain that 
\begin{equation}
\lim_{\delta \rightarrow 0 }  n_{\sharp,\delta}(u, w_h) =  n_{\sharp}(u, w_h).  \label{eq:identity_n_sharp}
\end{equation}
The next step of the proof consists of showing that 
\begin{equation}
\lim_{\delta \rightarrow 0} \sum_{F \in \mathcal{F}_{\Gamma,h} \cup \Gamma_e} \int_F [w_h \mathcal K_{\delta}^d(\bm \sigma(u))] \cdot \bm n_F   = \int_{\Gamma}\left( C_M \frac{\partial [u]}{\partial t}[w_h] + f_{\Gamma}([u]) \right)[w_h].  \label{eq:interface_boundary_terms}
\end{equation}
The details of deriving \eqref{eq:interface_boundary_terms} are provided in \Cref{appendix:interface_cond}. 

The identities  \eqref{eq:identity_n_sharp} and \eqref{eq:interface_boundary_terms} yield 
\begin{multline}
    \lim_{\delta \rightarrow 0 }  \del[3]{ n_{\sharp,\delta}(u, w_h) + \sum_{F \in \mathcal{F}_{\Gamma,h} \cup \Gamma_e } \int_F [w_h \mathcal K_{\delta}^d(\bm \sigma(u))] \cdot \bm n_F  } \\  =  n_{\sharp}(u, w_h)+ \int_{\Gamma}\left( C_M \frac{\partial [u]}{\partial t} + f_{\Gamma}([u]) \right)[w_h].  \label{eq:consistency_lim_0} 
\end{multline} 
Further, since $\mathcal K_{\delta}^d(\bm \sigma (u))$ is smooth, we use the definition of the duality pairing \eqref{eq:def_duality_pair}, Green's theorem locally in the expression of \eqref{eq:def_eta_sharp_delta} and  the definition of $\epsilon_{K,F} = \bm n_F \cdot (\bm n_K)_{\vert_{F}}$. We  obtain  
\begin{align*}
n_{\sharp,\delta}(u, w_h) &= \sum_{F \in \mathcal{F}_h \backslash \mathcal{F}_{\Gamma,h}} \sum_{K \in \mathcal{T}_F}\int_{F} \frac12 \epsilon_{K,F} \mathcal{K}_{\delta}^d (\bm \sigma(u)) \cdot \bm n_K [w_h]  \\ & = \sum_{F \in \mathcal{F}_h \backslash \mathcal{F}_{\Gamma,h}} \int_{F} \{\mathcal{K}_{\delta}^d (\bm \sigma(u))\} \cdot \bm n_F [w_h].
\end{align*}
Since $\mathcal{K}_{\delta}^d(\bm  \sigma(u))$ is smooth when restricted to $\Omega_i$ or to $\Omega_e$, we have that $[\mathcal{K}_{\delta}^d (\bm \sigma(u))] \cdot \bm n_F = 0 $ on $F \in \mathcal{F}_h\backslash \mathcal{F}_{\Gamma,h}$. Therefore, 
\begin{multline*}
n_{\sharp,\delta}(u, w_h) + \sum_{F \in \mathcal{F}_{\Gamma,h} \cup \Gamma_e } \int_F [w_h \mathcal K_{\delta}^d(\bm \sigma(u))] \cdot \bm n_F  \\  =   \sum_{F \in \mathcal{F}_h \cup \Gamma_e} \int_F [w_h \mathcal K_{\delta}^d(\bm \sigma(u))] \cdot \bm n_F   = \sum_{K \in \mathcal{T}_h} \int_{\partial K} \mathcal K_{\delta}^d(\bm \sigma(u)) \cdot \bm{n}_K w_h{}_{\vert K}. 
\end{multline*}
Using Green's theorem locally in each element and the commutativity properties of the operators  $\mathcal{K}_{\delta}^d$ and $\mathcal{K}_{\delta}^b$ \cite{ern2016mollification}, we find that 
\begin{multline}
n_{\sharp,\delta}(u, w_h) + \sum_{F \in \mathcal{F}_{\Gamma,h} \cup \Gamma_e} \int_F [w_h \mathcal K_{\delta}^d(\bm \sigma(u))] \cdot \bm n_F  \\ = \sum_{K \in \mesh} \int_K (\mathcal{K}_{\delta}^d(\bm \sigma(u)) \cdot \nabla w_h + \mathcal{K}_{\delta}^b(\nabla \cdot \bm \sigma (u)) w_h). 
\end{multline}
Passing to the limit as $\delta \to 0$ and  using \eqref{eq:consistency_lim_0} along with the convergence properties of  $\mathcal{K}_{\delta}^d$ and $\mathcal{K}_{\delta}^b$ yields 
\begin{multline}
  \int_{\Gamma}\left( C_M \frac{\partial [u]}{\partial t}[w_h] + f_{\Gamma}([u]) \right)[w_h] +   n_{\sharp}(u, w_h) \\ = \sum_{K \in \mesh} \int_K \bm \sigma(u) \cdot \nabla w_h  + \sum_{K \in \mesh} \int_{K} \nabla \cdot (\bm \sigma (u) ) w_h. 
\end{multline}
Using that $\nabla \cdot ( \bm \sigma (u) ) = f $ in $L^2(\Omega)$ since  $u$ solves \eqref{eq:EMI} concludes the proof.
\end{proof}
To proceed in the error analysis, we estimate the consistency error. Consider the Scott-Zhang interpolant~\cite{scott1990finite} operators  $\mathcal{I}_h^i: H^1(\Omega_i) \rightarrow V^1_{h,i} \cap H^1(\Omega_i)$  defined over $\mathcal{T}_{h, i} = \mathcal{T}_h \cap \Omega_i$ and  $\mathcal{I}_h^e: H^1(\Omega_e) \rightarrow V^1_{h,e} \cap H^1(\Omega_e)$ defined over $\mathcal{T}_{h,e} = \mathcal{T}_h \cap \Omega_e$. We refer to~\cite{ciarlet2013analysis,scott1990finite} for the estimates in fractional Sobolev spaces used hereinafter.  For $u \in H^1(\Omega_i \cup \Omega_e)$, we define $\mathcal{I}_h u  = \mathcal{I}_{h,i} (u{}_{ \vert_{\Omega_i}})$ in $\Omega_i$ and $\mathcal I_h u = \mathcal{I}_{h,e} (u{}_{\vert_{\Omega_e}})$ in $\Omega_e$. Clearly, $\mathcal I_h u \in H^1(\Omega_i \cup \Omega_e) \cap V_h^1$.  
\begin{lemma}[Consistency error] \label{lemma:bounding_W3}
Let $u \in V^{1+s}(\Omega) \cap W$. For any $\eta_h \in V_h^k$, we have the bound 
\begin{multline}
    \del[3]{\eta_\sharp (u, \eta_h) +  \sum_{K \in \mesh} \int_K  \kappa \nabla u \cdot  \nabla \eta_h - a_h(\mathcal{I}_h u , \eta_h) } \\    \lesssim h^s  (\|u\|_{H^{1+s}(\Omega_i \cup \Omega_e)} + h^{1-s} \|f\|_{L^2(\Omega)} )\,\, |\eta_h|_{\DG} .
\label{eq:bound_W3}
\end{multline}
\end{lemma}
\begin{proof}
Recall that $[u] = [\mathcal{I}_h u ] =  0 $ on  $F \in \mathcal{F}_h \backslash \mathcal F_{\Gamma, h}$ since $u, \mathcal{I}_h u  \in H^1(\Omega_i \cup \Omega_e)$. Further, from ~\cite[Lemma 3.3 (a)]{ern2021quasi}, it follows that 
\begin{equation}
 \eta_{\sharp} (\mathcal{I}_h u , \eta_h) = - \sum_{F \in \mathcal{F}_h \backslash \mathcal{F}_{\Gamma,h}} \int_{F}   \{\kappa \nabla \mathcal{I}_h u\} \cdot \bm n_F [\eta_h] . 
\end{equation}
The above observations simplify the left hand side of \eqref{eq:bound_W3}, denoted here by $Q$, as follows: 
\begin{equation} \label{eq:def_Q}
Q  = \eta_\sharp(u - \mathcal{I}_h u , \eta_h) + \sum_{K\in \mathcal{T}_h} \int_{K}  \kappa \nabla(u- \mathcal{I}_h u) \cdot \nabla \eta_h. 
\end{equation}
 To estimate the first term, we use ~\cite[Lemma 3.2]{ern2021quasi} (with $p = \rho = 2d/(d-2s)$ and $q=2$) and the fact that $\nabla \cdot (\bm \sigma (u - \mathcal{I}_h u ))  = f $   since $\mathcal{I}_h u $ is a linear polynomial.   
We obtain for any $F \in \mathcal{F}_K$
    \begin{multline}
        |  \langle (\bm \sigma(u  -\mathcal{I}_h u ))_{\vert K} \cdot \bm n_{K} )_{\vert {F}} , [\eta_h] \rangle_{F} | \\   \lesssim  (h_K^s \|\bm \sigma (u) - \bm{\sigma}(\mathcal{I}_h u )\|_{L^\rho(K)} + h_K \|f \|_{L^2(K)})h_F^{-1/2} \|[\eta_h]\|_{L^2(F)}. 
    \end{multline}
Recalling the Sobolev embedding $H^{s}( K)^d  \hookrightarrow L^\rho (K)^d$ (with the correct scaling\footnote{We map to the  reference element, apply the Sobolev embedding, and  then map back to the physical element. See ~\cite[Lemma 11.7]{Ern:booki} for the norm scaling. For $v \in H^{s}(K)^d$, we estimate \begin{equation*}
    \|v\|_{L^\rho(K)} \lesssim h_K^{d/\rho} \|\hat v \|_{L^\rho(\hat K)} \lesssim h_K^{d/\rho} (\|\hat v\|_{L^2(\hat K)} + |\hat v|_{H^s(\hat K)}) \lesssim h_K^{d/\rho} (h_K^{-d/2}\|v\|_{L^2(K)} + h_{K}^{s-d/2} |v|_{H^s(K)}).  
\end{equation*} 
Using that $d/\rho -d/2 = -s$  and substituting $v = \bm \sigma (u - \mathcal{I}_h u)$  gives the bound. }), noting that $\kappa \in L^{\infty}(\Omega; \mathbb{R}^{d,d})$ and using approximation properties of $\mathcal{I}_h$, we have that 
\begin{align*}
h_K^s \|\bm \sigma (u - \mathcal{I}_h u )\|_{L^\rho(K)}  &\lesssim \| \nabla u - \nabla \mathcal{I}_h u \|_{L^2(K) } +h_K^s |u |_{H^{1+s}(K)} \\ 
& \lesssim h_K^s |u|_{H^{1+s}(\Delta_K)}, 
\end{align*}
where $\Delta_K$ is a macro-element. Note that since $\mathcal{I}_h u $ is defined locally on $\Omega_i$ or $\Omega_e$, $ \Delta_K \subset \Omega_i $ or $\Delta_K \subset \Omega_e$. It then follows that 
\begin{align*}
   |  \langle (\bm \sigma(u - \mathcal{I}_h u))_{\vert K} \cdot \bm n_{K} )_{\vert {F}} , [w_h] \rangle_{F} | &    \lesssim  ( h_K^s \|u\|_{H^{1+s}(\Delta_K)} + h_K \|f \|_{L^2(K)})h_F^{-1/2} \|[\eta_h]\|_{L^2(F)}. 
\end{align*}
 Thus, with Cauchy-Schwarz inequality and the definition of $\eta_{\sharp}$ \eqref{eq:def_eta_sharp}, we obtain that 
\begin{align}
|\eta_\sharp(u -\mathcal{I}_h u, \eta_h)| \lesssim h^s  (\|u\|_{H^{1+s}(\Omega_i \cup \Omega_e)} + h^{1-s} \|f\|_{L^2(\Omega)} )\,\, |\eta_h|_{\DG}. 
\end{align}
Combining the above with a simple application of the Cauchy-Schwarz inequality and standard approximation theory to bound the second term in \eqref{eq:def_Q} yields the estimate. 
    \end{proof}
\begin{theorem}[Convergence] \label{thm:convergence} Assume that $u \in L^2(0,T;V^{1+s}(\Omega) \cap W)$ with $\partial_t u \in L^2(0,T; H^{1/2+s}(\Omega_i \cup \Omega_e))$ for $s \in (0,1/2)$.  Further, assume that $ u(0) \in H^1(\Omega_i \cup \Omega_e)$. Then, the following estimate holds  
\begin{align}
    \frac{C_{\mathrm{coerc}}}{2} & \left( \int_{0}^T |u -u_h|^2_{\DG}\right)^{1/2}  + \frac{C_M}{2} \|[u-u_h]\|_{L^{\infty}(0,T;L^2(\Gamma))} \\ \nonumber & \lesssim h^{s } (  \|u\|_{L^2(0,T ;H^{1+s}(\Omega_i \cup \Omega_e))} + \|\partial_t u\|_{L^2(0,T;H^{1/2+s}(\Omega_i \cup \Omega_e))} )  \\ & \nonumber \quad + h(\|f\|_{L^2(0,T;L^2(\Omega)))}+  \|u^0\|_{H^1(\Omega_i \cup \Omega_e)}). 
\end{align}
\end{theorem}
\begin{proof}
The proof is based on Strang's second lemma. Define $\eta_h = u_h - \mathcal{I}_h u $. With the coercivity property \eqref{eq:coercivity}, we obtain  that 
\begin{equation}
C_{\mathrm{coerc}} |\eta_h|^2_{\DG} \leq  a_h(\eta_h, \eta_h) = a_h(u_h, \eta_h) - a_h( \mathcal{I}_h u, \eta_h). 
\end{equation} 
Using \eqref{eq:semi_discrete}, we have that
\begin{align}
C_{\mathrm{coerc}} |\eta_h|^2_{\DG} \leq   \int_{\Omega} f\eta_h - \int_{\Gamma} f_{\Gamma}([u_h])([\eta_h]) - \int_{\Gamma} C_{M} \frac{\partial [u_h]}{\partial t} [\eta_h]  - a_h(\mathcal{I}_h u , \eta_h). 
\end{align} 
Using \Cref{lemma:weak_consistency}, we obtain 
\begin{multline} \label{eq:error_eq}
C_{\mathrm{coerc}} |\eta_h|^2_{\DG} + \frac{C_M}{2} \frac{\mathrm d}{\mathrm{d} t} \|[\eta_h]\|^2_{L^2(\Gamma)} \\ \lesssim   C_M \int_{\Gamma} \frac{\partial [u - \mathcal{I}_h u]}{ \partial t}[\eta_h]  + \int_{\Gamma}(f_{\Gamma}([u]) - 
 f_{\Gamma}([u_h])) [\eta_h]  \\ + \del[3]{\eta_\sharp (u, \eta_h) +  \sum_{K \in \mesh} \int_K  \kappa \nabla u \cdot \nabla w_h - a_h(\mathcal{I}_h u, \eta_h) } :=  W_1 + W_2 + W_3.
\end{multline}
To bound $W_1$ and $W_2$, consider $v \in H^{1/2+s}(\Omega_i \cup \Omega_e)$, $\mathcal{I}_h v = \mathcal{I}_{h,i} v\vert_{\Omega_i} $, and $\mathcal{I}_h v = \mathcal{I}_{h,e} v\vert_{\Omega_e}$. These operators are well defined over $H^{1/2+s}(\Omega_i \cup \Omega_e)$~\cite[Proposition 2.1]{ciarlet2013analysis}. Then, for any $F \in \Gamma_h, F = \partial K_1 \cap \partial K_2$ with $K_1 \subset \Omega_i, K_2 \subset \Omega_e$,  we employ the trace inequality\footnote{For $F \subset \partial K$, we map to the reference element, apply the trace estimate $H^{1/2+s}(\hat K) \hookrightarrow L^2(\hat F)$, and map back  
\begin{multline*}
\|v\|_{L^2(F)} \lesssim h_F^{(d-1)/2} \|\hat v\|_{L^2(\hat F)} \lesssim h_F^{(d-1)/2} ( \|\hat v\|_{L^2(\hat K)} + |\hat v |_{H^{1/2+s}(\hat K)}) \\ \lesssim h_F^{(d-1)/2} ( h_K^{-d/2} \| v\|_{L^2( K)} + h_K^{1/2+s -d/2} | v |_{H^{1/2+s}( K)}) .
\end{multline*}
} followed by the approximation and stability properties in the $H^{1/2+s}$-norm of the Scott-Zhang interpolant~\cite[Theorem 3.3]{ciarlet2013analysis}
\begin{align*}
\|[v - \mathcal{I}_h v]\|_{L^2(F)} &\lesssim h_{F}^{-1/2}\|v - \mathcal{I}_h v\|_{L^2(K_1 \cup K_2)} + h_F^{s} | v -\mathcal{I}_h v|_{H^{1/2  + s}(K_1 \cup K_2)} \\ 
& \lesssim h^{s} (\|v\|_{H^{1/2+s} (\Delta_{K_1})} + \|v\|_{H^{1/2+s}(\Delta_{K_2})}) .
\end{align*}
 Thus, summing over all the faces yields 
\begin{equation}
\|[v- \mathcal{I}_h v]\|_{L^2(\Gamma)}
\lesssim h^{s} \|v\|_{H^{1/2+s}(\Omega_i \cup \Omega_e)}.  \label{eq:app_bound_sh}
\end{equation}
Hence, applying the Cauchy-Schwarz inequality, noting that $\partial_t \mathcal{I}_h u = \mathcal{I}_h \partial_t u $, and using  \eqref{eq:app_bound_sh}, we bound $W_1$ as follows 
\begin{equation}
W_1 \lesssim \|[\partial_t u - \mathcal{I}_h (\partial_t u) ]\|_{L^2(\Gamma)} \|[\eta_h]\|_{L^2(\Gamma)} \lesssim h^{s} \|\partial_t u\|_{H^{1/2+s}(\Omega_i \cup \Omega_e)} \|[\eta_h]\|_{L^2(\Gamma)}. 
\end{equation}
The term $W_2$ is bounded by applying the Lipschitz continuity bound of $f_{\Gamma}$, triangle inequality, and \eqref{eq:app_bound_sh} (here we can take $s=1/2$). We have that 
\begin{equation}
W_2 \lesssim \|[u_h - u]\|_{L^2(\Gamma)} \|[\eta_h]\|_{L^2(\Gamma)} \lesssim \|[\eta_h ]\|^2_{L^2(\Gamma)} + h^{1/2} \|u\|_{H^1(\Omega_i \cup \Omega_e)} \|[\eta_h]\|_{L^2(\Gamma)}. 
\end{equation}
The term $W_3$ is bounded in \Cref{lemma:bounding_W3}. Collecting the bounds for $W_1$, $W_2$ and $W_3$ in \eqref{eq:error_eq} and appropriately applying Young's inequality yield 
\begin{multline}
\frac{C_{\mathrm{coerc}}}{2} |\eta_h|^2_{\DG} + \frac{C_M}{2} \frac{\mathrm d}{\mathrm{d} t} \|[\eta_h]\|^2_{L^2(\Gamma)} \\  \lesssim h^{2s }( \|u\|^2_{H^{1+s}(\Omega_i \cup \Omega_e)} +\|\partial_t u\|^2_{H^{1/2+s}(\Omega_i \cup \Omega_e)} +  h^{2-2s}\|f\|^2_{L^2(\Omega)} )  + \|[\eta_h]\|^2_{L^2(\Gamma)}. 
\end{multline}
We integrate  the above over time and use continuous Gronwall's inequality to obtain
\begin{multline}
   \frac{C_{\mathrm{coerc}}}{2} \int_{0}^t |\eta_h|^2_{\DG} + \frac{C_M}{2} \|[\eta_h](t)\|^2_{L^2(\Gamma)} \\  \lesssim h^{2s } \int_{0}^{t}( \|u\|^2_{H^{1+s}(\Omega_i \cup \Omega_e)} + \|\partial_t u\|^2_{H^{1/2+s}(\Omega_i \cup \Omega_e)} +  h^{2-2s}\|f\|^2_{L^2(\Omega)} )  + \|[\eta_h(0)]\|^2_{L^2(\Gamma)}. 
\end{multline}
Finally, we note that \[ [\eta_h(0)] = [u_h^0] - [\mathcal{I}_h u^0] = \pi_h \hat u^0 - \hat u^0 + [u^0 - \mathcal{I}_h u^0], \]
since $[u_h^0] = \pi_h \hat u^0$ and $[u^0] = \hat u^0 $.  Hence, 
$$\|[\eta_h(0)]\|_{L^2(\Gamma)} \lesssim \|\hat \pi_h u^0 - \hat u^0 \|_{L^2(\Gamma)} + \|[u^0 - \mathcal{I}_h u^0]\|_{L^2(\Gamma)} \lesssim h^{1/2} \|u^0\|_{H^1(\Omega_i \cup  \Omega_e)},$$ 
where we used the approximation properties of the $L^2$ projection on $\Gamma$. The result is then concluded by standard applications of the triangle inequality and approximation theory. 
\end{proof}
\begin{remark}[Optimal rates] 
If the solution $u\in H^1(0,T;H^{k+1}(\Omega_i \cup \Omega_e))$, then one can show the optimal error rate of order $k$  in the energy norm. In fact, the proof of such an estimate would follow from standard arguments since strong consistency of the method holds. However, since in many applications, the domains $\Omega_i$ and $\Omega_e$  do not exhibit the standard assumptions (for e.g. convex and polygonal), such a regularity property is not to be expected. 
\end{remark}

\section{Fully discrete formulation}\label{sec:fullydiscrete}
Here, we consider a backward Euler discretization in time. Consider a uniform partition of the time interval $[0,T]$ into $N_T$ subintervals of length $\tau  < 1$. 
Hereinafter, we use the notation $g^n (\bm x) = g(t_n, \bm x) = g(n \tau , \bm x)$ for a given function $g$ and $1 \leq n \leq N_T$. 

Given $[u_h^0] = \hat \pi_h \hat u^0$, the backward Euler interior penalty DG method reads as follows. Find $(u_h^n)_{1\leq n \leq N_T} \in V_h^k $ with $\int_{\Omega_e} u_h^n = 0 $ such that
\begin{multline}
\label{eq:b_e_dg}
\int_{\Gamma} C_M [u^n_h][v_h] + \tau  a_h( u^n_h, v_h)  \\ =  \int_{\Gamma} C_M [u^{n-1}_h][v_h] 
 + \tau  \int_{\Omega} f^n v_h  + \tau \int_{\Gamma} f_\Gamma ([u_h^{n-1}]) [v_h], \quad \forall v_h \in V_h^k. 
\end{multline}
\begin{lemma}[Well-posedness] For any $ 1 \leq n\leq N_T$, there exists a unique solution $u_h^n \in V_h^k$ with $\int_{\Omega_e} u_h^n = 0$ to \eqref{eq:b_e_dg}. In addition, 
\begin{equation}\label{eq:stability_estimates}
 \|[u_h^n]\|^2_{L^2(\Gamma)} 
+ \tau \sum_{\ell = 0}^n |u_h^\ell|_{\DG}^2  \lesssim   \|\hat u^0\|^2_{L^2(\Gamma)} +  \tau \sum_{\ell = 0}^n \|f^\ell \|_{L^2(\Omega)}^2 + \|f_{\Gamma} (0)\|_{L^2(\Gamma)}^2. 
\end{equation}
The above hidden constant depends on $T$ but is independent of $h$ and $\tau$.
\end{lemma}
\begin{proof}
First, observe that one can equivalently write \eqref{eq:b_e_dg} as a variational equality over all $v_h \in V_h^k$ with $\int_{\Omega_e} v_h = 0$. Indeed, this can be easily seen by noting that for any constant $C$, we have that $[C]=a_h(u_h^n, C) = \int_{\Omega} f^n C = 0$ by assumption on the data $f$.   The resulting system is a square linear system in finite dimensions, and it suffices to show uniqueness which follows from \Cref{lemma:poincare_discrete} after showing \eqref{eq:stability_estimates}.  To show \eqref{eq:stability_estimates},test \eqref{eq:b_e_dg} with $v_h = u_h^n$, use the Cauchy-Schwarz inequality, the coercivity estimate \eqref{eq:coercivity}, Poincar\'e's inequality \Cref{lemma:poincare_discrete}, and a similar estimate to \eqref{eq:bound_f_Gamma}. We obtain 
\begin{align*}
\frac{C_M}{2} & \left( \|[u_h^n]\|^2_{L^2(\Gamma)} -\|[u_h^{n-1}]\|^2_{L^2(\Gamma)} + \|[u_h^n - u_h^{n-1}]\|^2_{L^2(\Gamma)} \right) 
+ \frac{C_{\mathrm{coerc}}}{2} \tau |u_h^n|_{\DG}^2 \\ 
& \lesssim \tau \|f^n \|_{L^2(\Omega) }( |u_h^n|_{\DG} + \|[u_h^n]\|_{L^2(\Gamma)})  + \tau (\|[u_h^{n-1}]\|_{L^2(\Gamma)} + \|f_{\Gamma}(0)\|_{L^2(\Gamma)} ) \|[u_h^n]\|_{L^2(\Gamma)}.
\end{align*}
We further bound $\|[u_h^n]\|_{L^2(\Gamma)} \leq \|[u_h^n - u_h^{n-1}]\|_{L^2(\Gamma)} + \|[u_h^{n-1}]\|_{L^2(\Gamma)} $. With appropriate applications of Young's inequality, we obtain that 
\begin{align*}
\frac{C_M}{2} & \left( \|[u_h^n]\|^2_{L^2(\Gamma)} -\|[u_h^{n-1}]\|^2_{L^2(\Gamma)} + (1-\tau) \|[u_h^n - u_h^{n-1}]\|^2_{L^2(\Gamma)} \right)   
+ \frac{C_{\mathrm{coerc}}}{4} \tau |u_h^n|_{\DG}^2 \\ 
& \lesssim  \tau \|[u_h^{n-1}]\|^2_{L^2(\Gamma)} +  \tau \|f^n\|_{L^2(\Omega)}^2 + \tau \|f_{\Gamma} (0)\|_{L^2(\Gamma)}^2. 
\end{align*}
Summing the above from  $n=1,\ldots, N_T$,  using discrete Gronwall inequality and noting that $\|[u_h^0]\|_{L^2(\Gamma)} = \|\hat \pi_h \hat u^0\|_{L^2(\Gamma)} \leq \|\hat u^0\|_{L^2(\Gamma)}$ yields the result.
\end{proof}

\begin{theorem}[Error estimate] 
Assume that the conditions of \Cref{thm:convergence} are satisfied. Further, assume that $\partial_t u \in L^2(0,T;H^{1+s} (\Omega_i \cup \Omega_e))$ and that $ [u] \in H^2(0,T;L^2(\Gamma))$. Then, for any $ 1\leq \ell \leq N_T$,  the following estimate holds.  
\begin{multline}
\|[u_h^n - u^n]\|_{L^2(\Gamma)}^2 + \tau \sum_{\ell = 1}^n  |u_h^\ell - u^\ell|_{\DG}^2  \lesssim \tau^2 \|\partial_{t} [u] \|_{H^1(0,T; L^2(\Gamma))}  \\  
+ h^{2s} (\|u\|^2_{\ell^2(0,T;H^{1+s}(\Omega_i\cup \Omega_e) )} +\|\partial_t u\|^2_{\ell^2(0,T;H^{1/2+s}(\Omega_i\cup \Omega_e) )} + h^2 \|f\|_{\ell^2(0,T;L^2(\Omega))}^2). 
\end{multline}
\end{theorem}
\begin{proof}
We give a brief outline of the main steps and skip the details for brevity. Recall the definition of $\mathcal{I}_h$ given before \Cref{lemma:bounding_W3} and define $\rho_h^n = u_h^n - \mathcal{I}_h u^n$. Denote the consistency error introduced in Lemma \ref{lemma:bounding_W3} by  
\begin{equation}
    \mathcal{C}(u^n, v_h) = \eta_\sharp (u^n, v_h) +  \sum_{K \in \mesh} \int_K  \kappa \nabla u^n \cdot  \nabla v_h - a_h(\mathcal{I}_h u^n , v_h).
\end{equation}
With \Cref{lemma:weak_consistency}, we readily derive the following error equation. For all $v_h \in V_h$, 
\begin{multline}
\int_{\Gamma} C_M ([\rho_h^n] - [\rho_h^{n-1}]) [v_h] + \tau a_h (\rho_h^n , v_h) \\  = \int_{\Gamma} C_M [\mathcal{I}_h u^{n-1} - \mathcal{I}_h u^n + \tau (\partial_t u)^n][v_h] 
 +  \tau \mathcal{C}(u^n, v_h) + \tau \int_{\Gamma} (f_{\Gamma} [u_h^{n-1}] - f_{\Gamma} [u^n]) [v_h].  \label{eq:error_eq_1}
\end{multline}
We test the above with $v_h = \rho_h^n$ and we denote the terms on right-hand side of the above equations by $T_1$, $T_2$ and $T_3$, respectively.
We bound $T_1$ with standard estimates:
\begin{align*}
T_1  &\lesssim  (\tau^{3/2} \|\partial_{tt} [u] \|_{L^2(t^{n-1},t^n; L^2(\Gamma))} + \tau^{1/2} \|\partial_t [\mathcal{I}_h u - u]\|_{L^2(t^{n-1},t^n; L^2(\Gamma))} ) \|[\rho_h^n]\|_{L^2(\Gamma)}    \\ \nonumber & \lesssim  \tau^2  \|\partial_{tt} [u] \|_{L^2(t^{n-1},t^n; L^2(\Gamma))}^2 + h^{2s} \|\partial_t u \|^2_{L^2(t^{n-1},t^n; H^{1/2+s}(\Omega))} + \tau \|[\rho_h^{n-1}]\|_{L^2(\Gamma)}^2 + \frac{C_M}{8} \tau \|[\rho_h^{n}- \rho_h^{n-1}]\|_{L^2(\Gamma)}^2 
\end{align*} 
In the above, we also used \eqref{eq:app_bound_sh}. The last term in \eqref{eq:error_eq_1} is bounded by the Lipschitz continuity of $f_{\Gamma} $ and by writing $u_h^{n-1} - u^n = \rho_h^{n-1} + (\mathcal{I}_h u^{n-1} -u^{n-1}) + (u^{n-1} - u^n)$. 
\begin{multline}
T_3 
\lesssim \frac{C_M}{4}\tau  \|[\rho_h^n-\rho_h^{n-1}]\|^2_{L^2(\Gamma)} + \tau \|[\rho_h^{n-1}]\|_{L^2(\Gamma)}^2 \\ + \tau h \|u^{n-1}\|^2_{H^1(\Omega_i \cup \Omega_e)} + \tau \|[\partial_t u ]\|_{L^2(t^{n-1},t^n;L^2(\Gamma))}^2. 
\end{multline} 
The consistency error is bounded in \Cref{lemma:bounding_W3}. Collecting the above bounds, summing over $n$, and utilizing Gronwall's inequality yields the estimate.
\end{proof}

\section{Numerical results}\label{sec:numerics}
Here we present a series of numerical examples demonstrating different
properties of the DG method for the EMI model. First, error convergence of the
discretization is verified. \Cref{ex:onecell} and \Cref{ex:lowreg} consider a single cell
setting \eqref{eq:EMI} which was analyzed theoretically in previous sections.
In \Cref{ex:twocell}, we show that the DG scheme easily
extends to the EMI model of multiple cells in contact. \Cref{ex:mg} demonstrates that 
the resulting linear systems can be robustly solved using algebraic multigrid
preconditioners. 
{Finally, as a step towards realistic applications, the DG discretization of the EMI model is 
used to study stimulus propagation in a sheet of 3D cells in \Cref{ex:eirill}.}
In all the examples,  we apply the symmetric interior penalty
DG method, i.e. $\epsilon=1$ in \eqref{eq:poisson_DGk}, with the stabilization parameter $\gamma=20$. 


\begin{example}[Single cell]\label{ex:onecell}
\normalfont
We let $\Omega=(0, 1)^2$ and $\Omega_i$ be a plus-shaped cell 
  with width and height 0.75 and consider  
  \begin{equation}\label{eq:onecell}
  \begin{aligned}
      u_e = \sin\left(\pi(x+y)\right)\exp\left(-\omega t\right),\quad
      u_i = \cos\left(2\pi(x-y)\right))
  \end{aligned}
  \end{equation}
  as the solution of \eqref{eq:EMI} where the source terms $f_e$, $f_i$ and
  $f_{\Gamma}$ are computed for coefficients $\kappa_e=1$, $\kappa_i=2$, $C_M=1$ and
  $\omega=10^{-6}$. We note that \eqref{eq:onecell} requires an additional source term
  in the interface flux condition, i.e. $\kappa_i\nabla u_i\cdot\bm n-\kappa_e\nabla u_e\cdot\bm n=g_{\Gamma}$
  on $\Gamma \times (0, T)$, as well as in the Neumann boundary condition ,
  i.e. $-\kappa_e\nabla u_e\cdot\bm n_e=g_{\Gamma_e}$ on $\Gamma_e\times (0, T)$.
  The geometry together with unstructured initial mesh used in the uniform mesh refinement
  and the solution at time $t=0$ are shown in \Cref{fig:onecell}. 
  \begin{figure}[ht]
    \begin{center}
\includegraphics[height=0.225\textwidth]{./img/tikz_domain_AinvluLscale1eps_value0.5get_eigenvalues-1k1_value1k2_value2nrefs1pdegree1poincare_scale1preconditionera.pdf}
      \hspace{5pt}
    \includegraphics[height=0.225\textwidth]{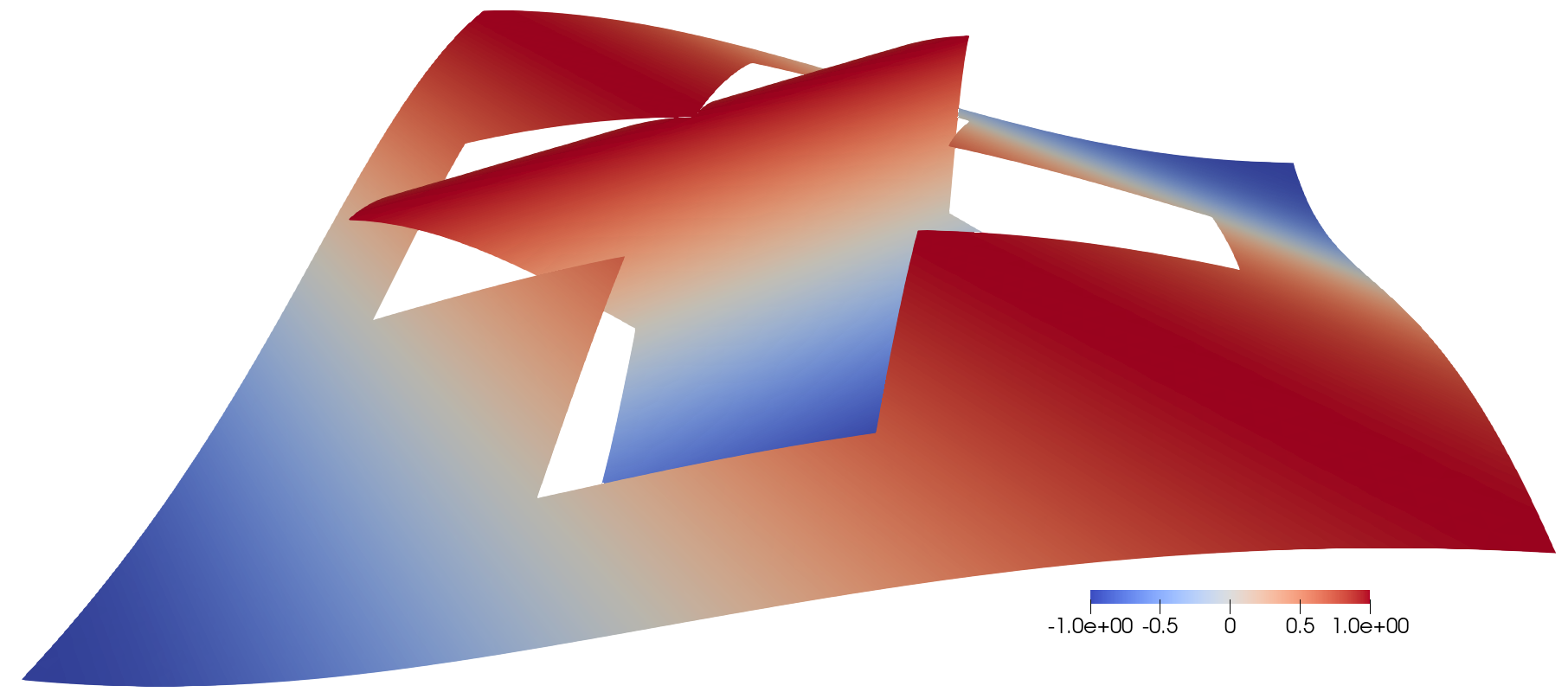}
    \end{center}
    \caption{
      Single cell EMI model \eqref{eq:EMI}.
      Problem geometry together with initial mesh (left) and solution at $t=0$ (right)
      of the test problem in \Cref{ex:onecell}.
    }
    \label{fig:onecell}
  \end{figure}

  With timestep size $\tau=h/10$ we evaluate spatial errors of the backward Euler
  discretization of \eqref{eq:EMI} at the final time $T=10^{-2}$. For polynomial
  degrees $k=1, 2, 3$,  \Cref{tab:onecell} shows that the proposed DG discretization
  yields the expected convergence rates of order $k$ in the DG-seminorm \eqref{eq:dgseminorm}
  while order $k+1$ is observed in the $L^2$-norm.
  \begin{table}[ht]
    \begin{center}
      \scriptsize{
        \renewcommand{\arraystretch}{1.1}
        \setlength{\tabcolsep}{3pt}
        \begin{tabular}{c|c|c|c|c|c|c}
          \hline
          \multirow{2}{*}{$l$} & \multicolumn{3}{c|}{$\lvert u - u_h \rvert_{\mathrm{DG}}$} & \multicolumn{3}{c}{$\lVert u - u_h \rVert_{L^2}$}\\
          \cline{2-7}
          & $k=1$ & $k=2$ & $k=3$ & $k=1$ & $k=2$ & $k=3$\\
          \hline
0 & 1.26e+00(--)   & 1.03e-01(--)   & 1.11e-02(--)   & 3.48e-02(--)   & 1.21e-03(--)   & 1.09e-04(--)   \\
1 & 5.62e-01(1.35) & 2.48e-02(2.37) & 1.09e-03(3.84) & 8.21e-03(2.4)  & 1.41e-04(3.57) & 5.16e-06(5.08) \\
2 & 2.78e-01(1.09) & 6.18e-03(2.15) & 1.54e-04(3.04) & 2.23e-03(2.02) & 1.78e-05(3.21) & 3.79e-07(4.05) \\
3 & 1.44e-01(1.01) & 1.72e-03(1.97) & 2.24e-05(2.97) & 6.31e-04(1.94) & 2.63e-06(2.94) & 2.90e-08(3.95) \\
          \hline
        \end{tabular}
      }
    \end{center}
    \caption{
      Spatial errors of the backward Euler DG discretization of the EMI
      model \eqref{eq:EMI} setup in \Cref{ex:onecell}. Timestep size $\tau=h/10$ is
      used where $h$ is the mesh size at current refinement level $l$. The errors
      are evaluated at time $t=10^{-2}$. Estimated order of convergence is shown in the
      brackets.
    }
    \label{tab:onecell}
  \end{table}
\end{example}

We next consider approximation properties of the DG discretization for
EMI model with low regularity solutions.

\begin{example}[Low regularity]\label{ex:lowreg} 
\normalfont
With $\Omega_i=(0.25, 0.75)^2$
  we let $\Omega$ be an L-shaped domain obtained as complement of $(-1, 0)^2$
  in $(-1, 1)^2$, cf. \Cref{fig:lowreg}. Setting $\kappa_e=1$, $\kappa_i=2$, $C_M=1$ the
  data\footnote{
  As in \Cref{ex:onecell} we remark that the chosen solution requires an additional
  non-zero source term in order for the flux continuity on the interface to hold.
  } for \eqref{eq:EMI} are constructed based on the solution
  \begin{equation}\label{eq:lr}
    u_e = (1+t)r^{s}\sin(s\theta), u_i = u_e + \left(1+t\right)\left(r\sin(\theta) + r\cos(\theta)\right),
  \end{equation}
  where $r, \theta$ are the polar coordinates in $xy$-plane and $0<s<1$. We note that
  $u_e(\cdot, t)\in H^{1+s}(\Omega_e)$.
  \begin{figure}[ht]
    \begin{center}
\includegraphics[height=0.225\textwidth]{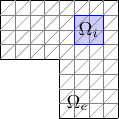}
      \hspace{5pt}
    \includegraphics[height=0.225\textwidth]{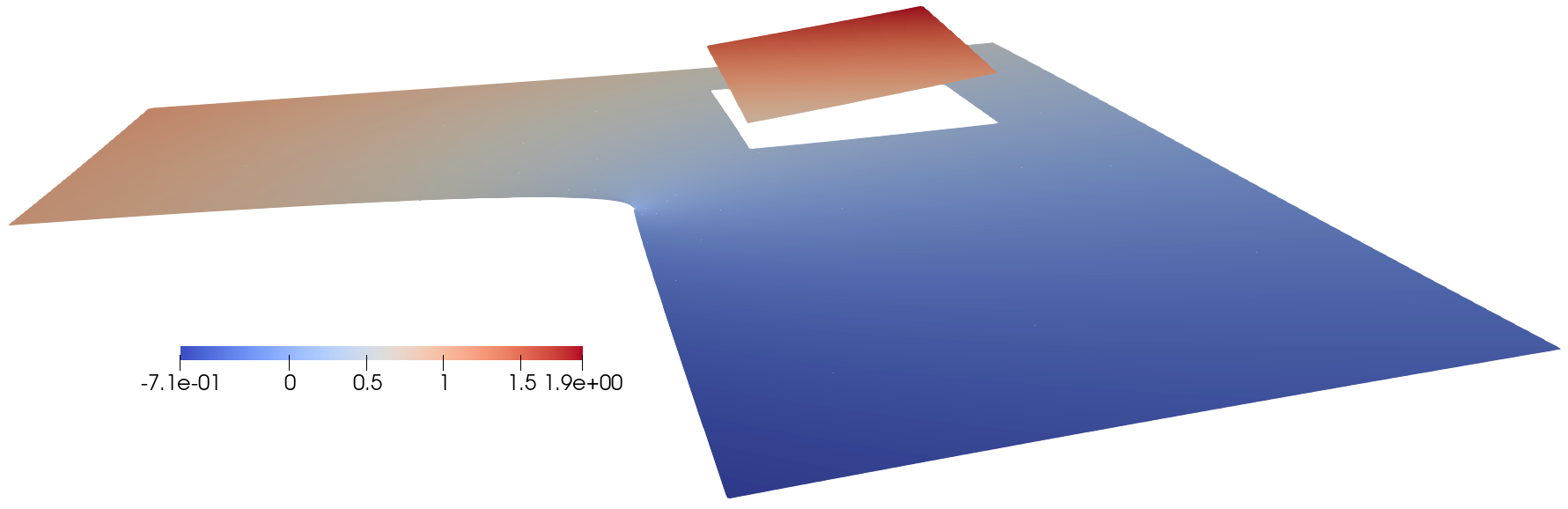}
    \end{center}
    \caption{
      Problem domain with structured initial mesh (left) and the low regularity (here \eqref{eq:lr} with $s=0.5$ is shown)
      solution of the EMI model \eqref{eq:EMI} at $t=0$ (right).
    }
    \label{fig:lowreg}
  \end{figure}

  Setting $\tau=10^{-5}$ and $T=10\tau$ and considering, for simplicity, \eqref{eq:EMI}
  with Dirichlet boundary conditions, the problem geometry is discretized by structured
  meshes (in order to obtain more stable convergence rates). Using linear elements \Cref{tab:lowreg}
  then confirms the expected convergence rate of order $s$ in the DG-seminorm \eqref{eq:dgseminorm}.
  \begin{table}[h]
    \begin{center}
      \scriptsize{
        \renewcommand{\arraystretch}{1.1}
        \setlength{\tabcolsep}{3pt}
        \begin{tabular}{c|c|c|c|c|c|c}
          \hline
          \multirow{2}{*}{$l$} & \multicolumn{3}{c|}{$\lvert u - u_h \rvert_{\mathrm{DG}}$} & \multicolumn{3}{c}{$\lVert u - u_h \rVert_{L^2}$}\\
          \cline{2-7}
          & $s=0.25$ & $s=0.5$ & $s=0.75$ & $s=0.25$ & $s=0.5$ & $s=0.75$\\
          \hline
0 & 1.55E-01(--)   & 1.39E-01(--)   & 7.35E-02(--)   & 8.47E-03(--)   & 9.30E-03(--)   & 5.90E-03(--)  \\
1 & 1.31E-01(0.24) & 1.00E-01(0.48) & 4.57E-02(0.69) & 3.85E-03(1.14) & 3.86E-03(1.27) & 2.20E-03(1.42)\\
2 & 1.11E-01(0.25) & 7.13E-02(0.49) & 2.80E-02(0.71) & 1.76E-03(1.13) & 1.59E-03(1.28) & 8.06E-04(1.45)\\
3 & 9.32E-02(0.25) & 5.06E-02(0.49) & 1.70E-02(0.72) & 8.18E-04(1.10) & 6.60E-04(1.27) & 2.93E-04(1.46)\\
4 & 7.84E-02(0.25) & 3.59E-02(0.50) & 1.02E-02(0.73) & 3.89E-04(1.07) & 2.77E-04(1.25) & 1.07E-04(1.46)\\
5 & 6.59E-02(0.25) & 2.54E-02(0.50) & 6.16E-03(0.74) & 1.89E-04(1.04) & 1.18E-04(1.23) & 3.90E-05(1.45)\\
6 & 5.54E-02(0.25) & 1.80E-02(0.50) & 3.69E-03(0.74) & 9.42E-05(1.01) & 5.06E-05(1.22) & 1.43E-05(1.45)\\
\hline
        \end{tabular}
      }
    \end{center}
    \caption{
      Spatial errors of the backward Euler DG discretization of the EMI
      model \eqref{eq:EMI} with low regularity $H^{1+s}$ solution setup in \Cref{ex:lowreg}.
      Timestep size $\tau=10^{-5}$ and linear ($k=1$) elements were used. The errors were
      evaluated at final time $T=10\tau$. Estimated order of convergence is shown in the brackets.      
    }
    \label{tab:lowreg}
  \end{table}
\end{example}

The EMI model \eqref{eq:EMI} can be readily applied to study dynamics
of a single cell or collection of disconnected cells ~\cite{agudelo2013computationally}.
However, for studying systems of tightly packed cells the small spaces separating
them, e.g. the gap junctions in cardiac tissue ~\cite{rohr2004role}, can
be modeled by placing the cells in contact and applying coupling conditions
on the new interface ~\cite{tveito2017cell}. We note that the combined interface
is then no longer a smooth manifold, cf. \Cref{fig:twocell}. As an example
of a system with gap junctions, we next consider the EMI model of two cells in contact
and its DG discretization.

\begin{example}[Cells in contact]\label{ex:twocell}
\normalfont
  The EMI model \eqref{eq:EMI} is extended to a system with $N$ non-overlapping cells
  following \cite{tveito2017cell}. The cells occupying subdomains
  $\Omega_i\subset\Omega$, $i=1, \dots, N$, $\Omega_i\cap\Omega_j=\emptyset$ {for $i\neq j$}  are contained in a bounded polygonal domain $\Omega \subset\mathbb{R}^d$, $d=2, 3$.
  Letting $\Omega_0 \equiv \Omega_e = \Omega \backslash \bigcup_{1\leq i\leq N} \Omega_i$
  denote the extracellular domain, we define interfaces $\Gamma_{(i, j)}=\partial\Omega_i\cap\partial\Omega_j$, $0\leq i<j\leq N$
  and the set $\mathcal{I}=\left\{(i, j):\,\,\,0 \leq i < j \leq N, \lvert \Gamma_{(i, j)}\rvert > 0 \right\}$.
  With $\Gamma=\bigcup_{(i, j)\in\mathcal{I}}\Gamma_{(i, j)}$ we let $\Gamma_0 =\partial\Omega_0 \setminus\Gamma$ and denote
  by $\bm n_0$ its outer unit normal vector. On each interface $\Gamma_{(i, j)}$, $(i, j)\in\mathcal{I}$ we 
  define a jump operator $[u]_{(i, j)} = u_i - u_j$ and orient the unit normal vector $\bm n_{(i, j)}$ of $\Gamma_{(i, j)}$ as
  outer with respect to $\Omega_i$. We then define a jump operator on the global interface
  as
  \begin{equation}\label{eq:jump_many}
    [u]|_{\Gamma_{(i, j)}} = [u]_{(i, j)}\quad \forall (i, j)\in\mathcal{I},
  \end{equation}
   and finally consider the EMI model
\begin{equation}\label{eq:EMI_many}
  \begin{alignedat}{2}
    -\nabla\cdot(\kappa_i\nabla u_i) &= f_i   && \quad \text{ in }\Omega_i \times (0,T), i=0, 1, \dots, N,  \\
    \kappa_i\nabla u_i\cdot \bm n_{(i, j)}   &= \kappa_j \nabla u_j\cdot \bm n_{(i, j)}   && \quad \text{ on } \Gamma_{(i, j)}\times (0,T), (i, j)\in \mathcal{I},  \\
  C_M \frac{\partial [u]}{\partial t} + f_\Gamma ([u]) & =  -\kappa_i \nabla u_i \cdot \bm n_{(i, j)}  && \quad  \text{ on }  \Gamma_{(i, j)} \times (0,T),  (i, j)\in \mathcal{I}, \\
    \nabla u_0 \cdot \bm n_0  & = 0 &&  \quad \text{ on } \Gamma_0 \times (0,T),   \\ 
    [u](x,0) & = \hat u^0 && \quad \text{ on } \Gamma_{(i, j)}, (i, j)\in \mathcal{I}.  
  \end{alignedat}
\end{equation}
We assume that volumetric sources $f_i \in L^2(0,T;L^2(\Omega_i))$,
symmetric positive definite conductivities $\kappa_i \in L^\infty(\Omega_i, \mathbb{R}^{d,d})$,
membrane sources $f^{(i, j)}_{\Gamma}$, positive capacitances $C^{(i, j)}_M$ are given together with the 
initial conditions $\hat u_{(i, j)}^0 \in H^{1/2}(\Gamma_{(i, j)})$. We note that in \eqref{eq:EMI_many}
$f_{\Gamma}$ is defined through restriction as $f_{\Gamma}|_{\Gamma_{(i, j)}}=f^{(i, j)}_{\Gamma}$, and similarly
for $C_M$ and $\hat u^0$. 
Furthermore, as in \eqref{eq:EMI} a compatibility condition $\int_{\Omega}f=0$,
where $f|_{\Omega_i}=f_i$, is required due to the Neumann boundary condition on $\Gamma_0$.
  
With the space $V^k_h$ in \eqref{eq:discretesubspaces}, {the form $a_h$ given in \cref{eq:poisson_DGk}} (re)defined to reflect the
jump operator \eqref{eq:jump_many} and to {include all the facets on the interfaces $\Gamma_{i,j}$ in the definition of $\mathcal{F}_{\Gamma,h}$}, the semi-discrete DG approximation of the
EMI problem \eqref{eq:EMI_many} and the related fully discrete approximation
based on the backward Euler interior penalty DG are identical to those of the
single cell EMI problem stated respectively in \eqref{eq:semi_discrete} and
\eqref{eq:b_e_dg}.

To investigate the spatial convergence of the DG method for EMI model with gap junctions we
consider a system with two connected cells shown in \Cref{fig:twocell}. The problem
data are manufactured according to \eqref{eq:EMI_many} using
\begin{equation}
  \begin{aligned}
    u_0 = t\sin\left(\pi(x+y)\right),\,
    u_1 = t\cos\left(2\pi(x-y)\right),\,
    u_2 = t\sin\left(2\pi(x+y)\right)\\
  \end{aligned}
\end{equation}
and setting $\kappa_0=1$, $\kappa_1=2$, $\kappa_2=3$ and $C_M=1$. With time step
size $\tau=h/10$ the simulations are run until $T=1$. Spatial errors at the final
time for DG approximants with polynomial degrees $k=1, 2, 3$ are reported in \Cref{tab:twocell}. Here, as
in the single cell case, we observe convergence rates of order $k$ and $k+1$ in
DG-seminorm \eqref{eq:dgseminorm} and $L^2$-norm respectively.

  
   \begin{figure}
    \begin{center}
    \includegraphics[height=0.225\textwidth]{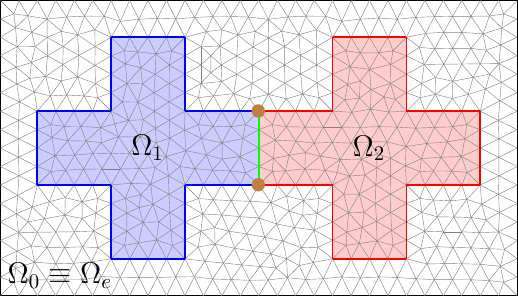}
    \hspace{5pt}
    \includegraphics[height=0.225\textwidth]{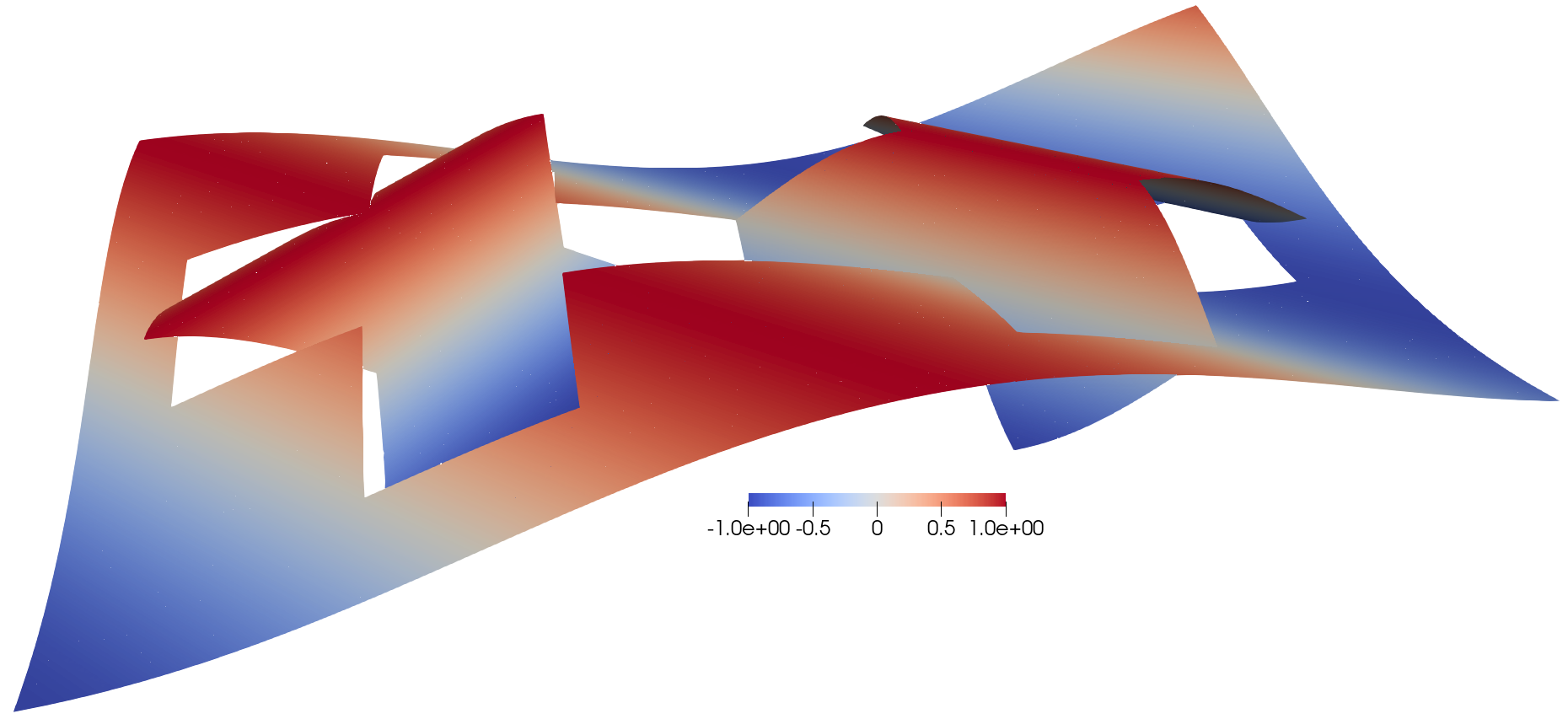}    
    \end{center}
    \caption{
      EMI model with gap junctions \eqref{eq:EMI_many}.
      Problem geometry together with initial mesh (left) and solution at $t=1$ (right)
      of the two-cell EMI test problem in \Cref{ex:twocell}. Extracellular
      interfaces $\Gamma_{(0, 1)}$, $\Gamma_{(0, 2)}$ are highlighted in blue and red.
      The interface $\Gamma_{(1, 2)}$ between cells $\Omega_1$ and $\Omega_2$ which idealizes a gap
      junction is shown in green. Brown nodes depict points of intersections of the three interfaces. 
    }
    \label{fig:twocell}
   \end{figure}

  \begin{table}
    \begin{center}
      \scriptsize{
        \renewcommand{\arraystretch}{1.1}
        \setlength{\tabcolsep}{3pt}
        \begin{tabular}{c|c|c|c|c|c|c}
          \hline
          \multirow{2}{*}{$l$} & \multicolumn{3}{c|}{$\lvert u - u_h \rvert_{\mathrm{DG}}$} & \multicolumn{3}{c}{$\lVert u - u_h \rVert_{L^2}$}\\
          \cline{2-7}
          & $k=1$ & $k=2$ & $k=3$ & $k=1$ & $k=2$ & $k=3$\\
          \hline
0 & 6.40e-01(--)   & 3.67e-02(--)   & 1.44e-03(--)   & 8.52e-03(--)   & 2.17e-04(--)   & 7.00e-06(--)  \\
1 & 3.26e-01(0.93) & 9.54e-03(1.86) & 1.83e-04(2.84) & 2.19e-03(1.87) & 2.79e-05(2.83) & 4.47e-07(3.79)\\
2 & 1.66e-01(0.87) & 2.46e-03(1.75) & 2.37e-05(2.64) & 5.70e-04(1.74) & 3.64e-06(2.63) & 2.93e-08(3.53)\\
3 & 8.34e-02(1.07) & 6.18e-04(2.15) & 2.95e-06(3.24) & 1.43e-04(2.15) & 4.53e-07(3.24) & 1.81e-09(4.33)\\
          \hline
        \end{tabular}
      }
    \end{center}
    \caption{
      Spatial errors of the backward Euler DG discretization of the EMI
      model \eqref{eq:EMI_many} setup in \Cref{ex:twocell}. Timestep size $\tau=h/10$ is
      used where $h$ is the mesh size at current refinement level $l$. The errors
      are evaluated at time $t=1$. Estimated order of convergence is shown in the
      brackets.
    }
    \label{tab:twocell}
  \end{table}
\end{example}


A consistent way of enforcing the constraint $\int_{\Omega_e} u_h^n = 0$ required
for uniqueness of the solution of the backward Euler DG approximation \eqref{eq:b_e_dg}
is provided by a Lagrange multiplier method. In \Cref{rmrk:poincare} we next
discuss design of robust preconditioners for the resulting saddle point problem.

\begin{remark}[Robust preconditioning]\label{rmrk:poincare} The variational
  problem due to the  backward Euler DG approximation \eqref{eq:b_e_dg} and to the 
  Lagrange multiplier method for the constraint $\int_{\Omega_e} u_h^n = 0$
  reads: Find $u_h\in V^k_h$ and $p_h\in Q_h=\mathbb{R}$ such that
  \begin{equation}\label{eq:lm}
    \begin{aligned}
      &\int_{\Gamma} \frac{C_M}{\tau} [u^n_h][v_h] + a_h( u^n_h, v_h)  + \int_{\Omega_e}v_h p_h  =\\
&\int_{\Gamma} \frac{C_M}{\tau} [u^{n-1}_h][v_h] 
 + \int_{\Omega} f^n v_h  + \int_{\Gamma} f_\Gamma ([u_h^{n-1}]) [v_h], \quad\forall v_h\in V^k_h,\\ 
&\int_{\Omega_e} u^n_h q_h = 0, \quad\forall q_h\in\mathbb{R}.
   \end{aligned}
  \end{equation}
  By verifying the Brezzi conditions \cite{Boffi:book}, the system \eqref{eq:lm} can    
  be shown to be well-posed with the solution spaces $V^k_h$, $Q_h$ equipped with the norms
  \begin{equation}
  \begin{aligned}\label{eq:simple_dg}
    \lVert u_h \rVert_{V_h} = \left(\lVert u_h \Vert^2_{L^2(\Omega_e)} + \lvert u_h \rvert^2_{\mathrm{DG}} +
    \frac{C_M}{\tau} \lVert [u_h]\rVert^2_{L^2(\Gamma)}\right)^{1/2}\!\!\!\!,\, 
    \lVert q_h \rVert_{Q_h} = \lVert q_h \rVert_{L^2(\Omega_e)}.
  \end{aligned}
  \end{equation}
  Following the ideas of operator preconditioning ~\cite{mardal2011preconditioning, hiptmair2006operator}, a Riesz map
  with respect to \eqref{eq:simple_dg} defines a suitable (block-diagonal) preconditioner
  for \eqref{eq:lm}. However, with \eqref{eq:simple_dg} the constant in the Brezzi
  coercivity condition can be seen to depend on the Poincar{\'e}
  constant, cf. \Cref{lemma:poincare_discrete}. In turn, as illustrated in \Cref{tab:poincare}, performance
  of the preconditioner deteriorates on long and thin domains. At the same time,
  such domains are highly relevant in practical applications, e.g. simulations of
  signal propagation in neurons ~\cite{buccino2019does} through \eqref{eq:EMI} or in sheets of cardiac
  cells~\cite{tveito2017cell, huynh2023convergence} using \eqref{eq:EMI_many}.   

  \begin{figure}
    \centering
    \includegraphics[height=0.165\textwidth]{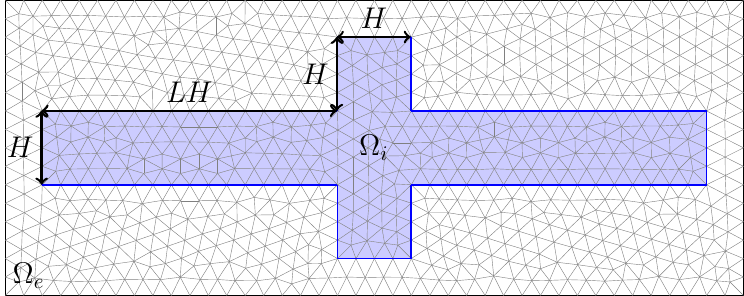}
    \hspace{5pt}    
    \includegraphics[height=0.165\textwidth]{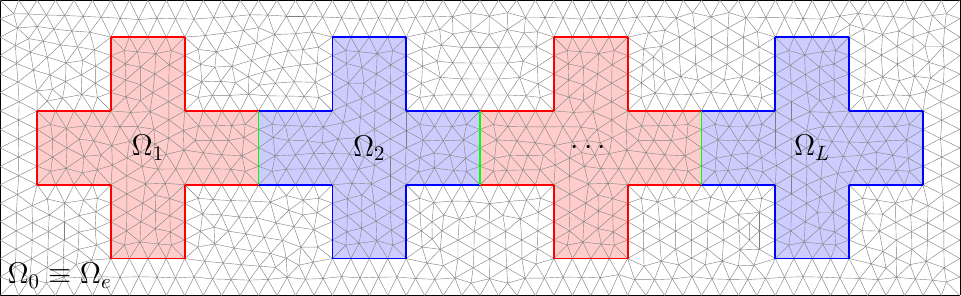}    
    \caption{
      Simplified problem geometries representing long and thin domains
      encountered in EMI applications. The domain diameter grows due to
      (left) the elongated shape of a single cell (modeling spatial characteristics
      of neurons, see \cite{buccino2019does}) or (right) stacking of cells in one direction
      (as is common in sheet simulations, e.g. \cite{tveito2017cell, huynh2023convergence}).
    }
    \label{fig:poincare}
  \end{figure}
  
  To address the domain sensitivity we consider the solution space $V^k_h\times \mathbb{R}$ with norms
  \begin{equation}
  \begin{aligned}\label{eq:poincare_dg}
    \lVert u_h \rVert_{V_h} = \left(\alpha^2\lVert u_h \Vert^2_{L^2(\Omega_e)} + \lvert u_h \rvert^2_{\mathrm{DG}}
    + \frac{C_M}{\tau} \lVert [u_h]\rVert^2_{L^2(\Gamma)}\right)^{1/2}\!\!\!\!,\, 
    \lVert q_h \rVert_{Q_h} = \frac{1}{\alpha}\lVert q_h \rVert_{L^2(\Omega_e)},
  \end{aligned}
  \end{equation}
  where $\alpha>0$ reflects the Poincar{\'e} constant. In particular, in the following,
  we use the approximation $\alpha\sim \text{diam}(\Omega_e)^{-1}$. Results reported in
  \Cref{tab:poincare} confirm that the preconditioner based on \eqref{eq:poincare_dg}
  provides uniform performance on the tested elongated domains.

  \begin{table}[h]
    \centering
    \scriptsize{
      \begin{tabular}{l|ccccc||ccccc}
       \hline
       \multirow{2}{*}{\backslashbox{norm}{$L$}} & \multicolumn{5}{c||}{Single cell} & \multicolumn{5}{c}{Connected cells}\\
       \cline{2-11}
       & 2 & 4 & 8 & 16 & 32 & 2 & 4 & 8 & 16 & 32\\
        \hline
        \eqref{eq:simple_dg}   & 19 & 19 & 20 & 23 & 28 & 19 & 28 & 39 & 59 & 98\\
        \eqref{eq:poincare_dg} & 23 & 23 & 23 & 23 & 23 & 19 & 19 & 19 & 19 & 19\\
      \hline
      \end{tabular}
    }
    \caption{
      Number of MinRes iterations required for solving \eqref{eq:lm} (with $k=1$)
      on domains in \Cref{fig:poincare} using different exact Riesz map preconditioners.
      Performance of \eqref{eq:simple_dg}-based preconditioner deteriorates on long domains.
      MinRes iterations start from a 0 initial vector and terminate once the preconditioned
      residual norm is below $10^{-12}$. 
    }
    \label{tab:poincare}
  \end{table}
  \end{remark}
  
\begin{example}[Multigrid realization of preconditioners]\label{ex:mg}
\normalfont
  Using single-cell EMI model \eqref{eq:EMI} with problem geometry from \Cref{ex:onecell}
  we finally study performance of inexact Riesz map preconditioners for linear systems due
  to backward Euler DG discretization \eqref{eq:b_e_dg}. Having addressed domain sensitivity
  in \Cref{rmrk:poincare} the focus here shall be on robustness of algebraic multigrid
  preconditioners with respect to time step size $\tau$. We note that for linear systems
  stemming from CG discretization of the EMI models \cite{budivsa2024algebraic} prove $\tau$-uniform
  convergence of algebraic multigrid. Therein custom block smoothers are used to handle the 
  interface between the subdomains in particular in case of strong coupling (due to
  small time step size).

  In the following, the action of a given Riesz map preconditioner is approximated
  by a single application of a V-cycle multigrid solver. We consider
  the classical algebraic multigrid (AMG) \cite{amg} 
  as implemented in
  the Hypre library \cite{hypre} and smoothed aggregation algebraic multigrid
  (SAMG) \cite{samg} with implementation provided by PyAMG \cite{pyamg2023}. Each
  solver is then applied to the linear operator inducing the norm \eqref{eq:poincare_dg}. In
  addition, we also consider preconditioning based on equivalent norms (cf. \eqref{eq:coercivity})
  which utilize the bilinear form $a_h$ in \eqref{eq:poisson_DGk}
  \begin{equation}
  \begin{aligned}\label{eq:poincare_a}
    \lVert u_h \rVert_{V_h} = \left(\alpha^2\lVert u_h \Vert^2_{L^2(\Omega_e)} + a_h(u_h, u_h) +
\frac{C_M}{\tau} \lVert [u_h]\rVert^2_{L^2(\Gamma)}\right)^{1/2}\!\!\!\!,\, 
    \lVert q_h \rVert_{Q_h} = \frac{1}{\alpha}\lVert q_h \rVert_{L^2(\Omega_e)}.
  \end{aligned}
  \end{equation}
  We note that the multiplier block of the preconditioners (being of size 1)
  is inverted exactly and only the leading block is realized by multigrid.

  The different preconditioners are compared in \Cref{fig:mg} in terms of
  iteration counts of the preconditioned MinRes solver. 
  Notably, irrespective of the Riesz map or
  the tested multigrid approximation the number of iterations appears bounded
  in the mesh and time step size.
  \begin{figure}[ht]
    \centering
    \includegraphics[width=\textwidth]{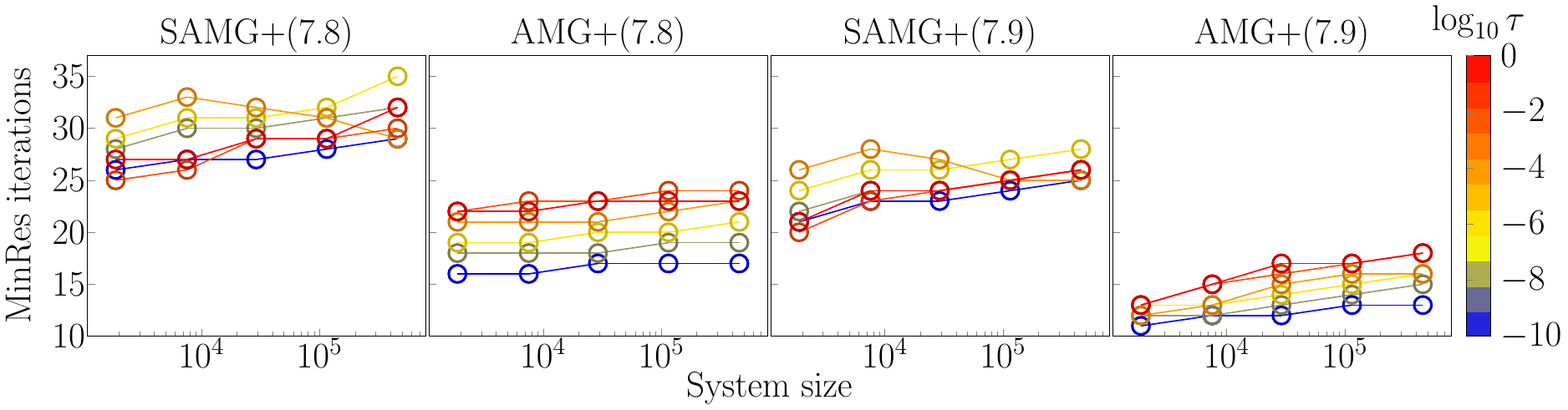}
    \caption{
      Performance of algebraic multigrid preconditioners for backward Euler
      DG ($k=1$) discretization \eqref{eq:b_e_dg} of single cell EMI model from \Cref{ex:onecell}.
      Number of MinRes iterations until convergence (based on relative error tolerance of $10^{-12}$)
      is shown for different mesh resolutions and time step sizes $\tau$. 
      Riesz maps based on norms \eqref{eq:poincare_dg} and \eqref{eq:poincare_a} are
      approximated using classical and smoothed aggregation algebraic multigrid solvers
      (AMG \cite{amg}, SAMG \cite{samg}).
    }
    \label{fig:mg}
  \end{figure}
\end{example}

\begin{example}[3D cell sheet]\label{ex:eirill}
  \normalfont
Motivated by tissue structures (often monolayers) studied in heart-on-chip
experiments, e.g.~\cite{grosberg2011ensembles}, we apply the EMI model \eqref{eq:EMI_many} to
a sheet of three-dimensional cardiac cells connected by gap junctions. 
The cell sheet consists of 15 cells, $5$ columns of cells in the $x$-direction and $3$ rows of
cells in the $y$-direction, using a single layer of cells in the $z$-direction. Each cardiac cell
is \textit{plus}-shaped in the $xy$-plane, connected through gap junctions both vertically
and horizontally. The cell sheet is visualized in \Cref{fig:geometry-3D}.
\begin{figure}
    \includegraphics[height=0.255\textwidth]{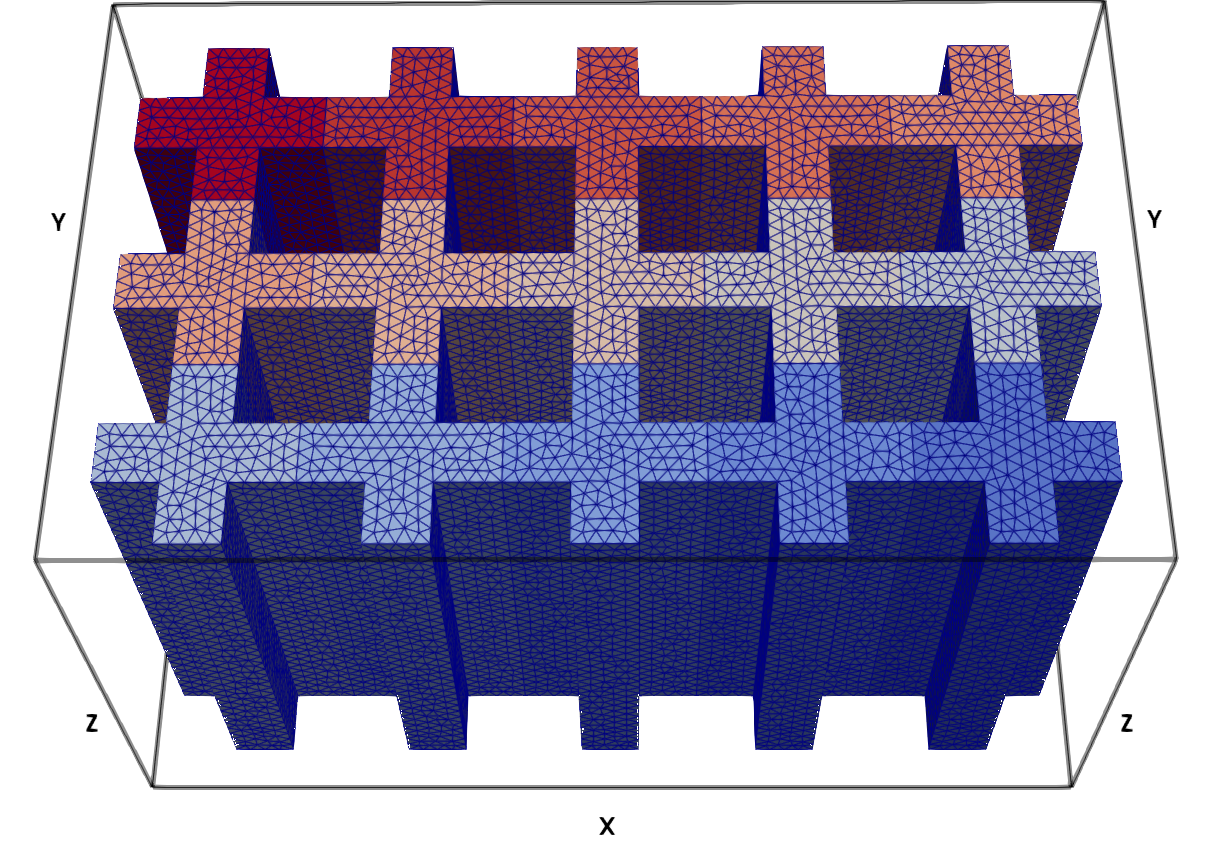}
    \hspace{5pt}
    \includegraphics[height=0.255\textwidth]{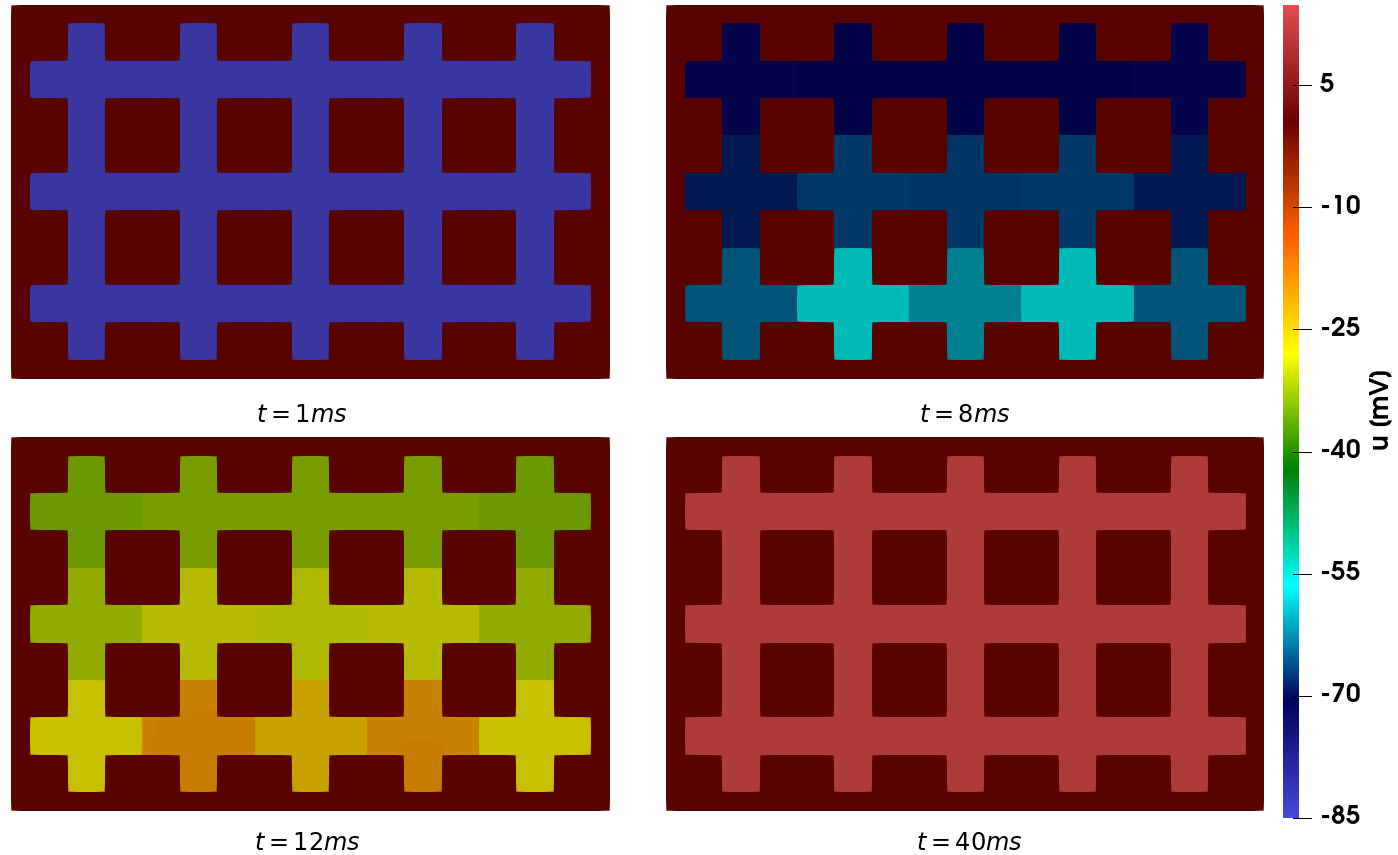}            
    \caption{
    Stimulus propagation in the sheet of 15 cardiac cells.
    (Left) Problem geometry with cells $\Omega_i$, $1\leq i \leq 15$ colored
    by $i$. (Right) Solution of the EMI model \eqref{eq:EMI_many} using the
    ODE model based on \cite{jaeger2019properties} at different time instants.
    The observed dynamics are due to time-dependent stimulus \eqref{eq:stimulus}
    applied to the second and fourth cells in the bottom row. Slice through
    the midplane is shown.
    }
    \label{fig:geometry-3D}
\end{figure}
Both the height and width of the cells were set to $\units{24 \cdot 10^{-4}}{cm}$, while
the length in the $z$-direction was $\units{64 \cdot 10^{-4}}{cm}$. The mesh resolution was
$4.3 \cdot 10^{-4}\,\mathrm{cm}$ which for DG discretization with linear elements led to
$1.9 \cdot 10^{6}$ degrees of freedom. 

The cardiac ODE model was inspired by ~\cite{jaeger2019properties}, reusing several
of their physiological parameters. This includes the intracellular and extracellular conductivities
in \eqref{eq:EMI_many}:
\begin{equation}
    \kappa_i =
    \begin{cases}
        \units{20}{mS / cm}
        & i = 0
        ,\\
        \units{4}{mS / cm}
        & 1 \leq i \leq 15,
    \end{cases}
\end{equation}
as well as the membrane capacitance and capacitance over the gap junctions:
\begin{equation}
    C_M^{(i, j)} =
    \begin{cases}
        \units{1}{\mu F/ cm^2}
        & (0, j)\in\mathcal{I}
        ,\\
        \units{0.5}{\mu F/ cm^2}
        & (i, j)\in\mathcal{I}, 0 < i < j.
    \end{cases}
\end{equation}
A passive current flows over the gap junctions, according to
\begin{equation}\label{eq:I_G}
    f^{(i, j)}_\Gamma \left([u] \right) = \frac{1}{R_G} \left([u] - E_G\right), 
    \quad \text{ on }
    \Gamma_{(i, j)}\times (0,T), 
    \quad (i, j)\in \mathcal{I},
    \quad 0 < i < j,
\end{equation}
with a resistance of 
$R_G = \units{0.05}{k \Omega cm^2}$ 
and the reversal potential over the gap junctions was set to 
$E_G = \units{0}{mV}$. 
The Aliev-Panfilov cardiac membrane model~\cite{aliev_simple_1996} was used on the intersection between the cells and the extracellular space, 
$\Gamma_{(0, j)}$, $(0, j) \in \mathcal{I}$. The model uses a unitless parameter $\alpha$ following the membrane potential according to
\begin{equation}\label{eq:alpha}
    \alpha \equiv \frac{[u] + E_M}{A}, 
    \quad \text{ on }\Gamma_{(0, j)} \times (0,T),
    \quad (0, j) \in \mathcal{I}, 
\end{equation}
where $E_M = \units{-85}{mV}$ is the reversal potential and $A = \units{100}{mV}$ determines the amplitude of the action potential. The membrane potential evolves in time with the unitless parameter $\beta$:
\begin{equation}\label{eq:aliev}
\begin{alignedat}{2}
    \frac{\partial [u]}{\partial t} &= c_t A
        \Big(I - k \alpha (\alpha - \eta_1)(\alpha - \eta_2) - \alpha\beta\Big), 
        \quad \text{ on }  \Gamma_{(0, j)} \times (0,T) , \,\,  (0, j) \in \mathcal{I}, 
\end{alignedat}
\end{equation}
where $\beta$ is governed by 
\begin{equation}
\frac{\partial \beta}{\partial t} =  - c_t 
\left(\mu_0 + \frac{\mu_1\beta}{\alpha + \mu_2}\right)
\Big(\beta  + k\alpha^2 - k\alpha(\eta_1 + \eta_2)\Big).
\end{equation}
Here, 
$c_t = \units{0.0775}{ms^{-1}}$
is a time constant, and the remaining constants were set to 
$k = 8$, 
$\eta_1 = 0.13$, 
$\eta_2 = 1$, 
$\mu_0 = 0.002$, 
$\mu_1 = 0.2$, and 
$\mu_3 = 0.3$.
The unitless time-dependent parameter $I$ determines the strength of the stimulus current.

The system is initially, at time $t_0 = \units{0}{ms}$, in its equilibrium state, i.e.
\begin{equation}\label{eq:Vm-equilibrium}
    [u](t_0) = 
    \begin{cases}
        \units{-85}{mV}&  
        \quad\text{ on } \Gamma_{(0, j)}, 
        \quad (0, j) \in \mathcal{I},
        \\
        \quad \ \units{0}{mV}& 
        \quad\text{ on } \Gamma_{(i, j)},
        \quad (i, j)\in \mathcal{I},
        \,\ 0 < i < j, 
    \end{cases}
\end{equation}
with $\beta(t_0) = 0$. To stimulate the system we evolve $I$ as 
\begin{equation}\label{eq:stimulus}
    I(t) = 
    \begin{cases}
        50 
        &\text{ for } 
        \units{5}{ms} < t < \units{15}{ms},\\
        0
        &\text{ otherwise}.
    \end{cases}
\end{equation}
However, the stimulus acts locally as only the second and fourth of the bottom cells were stimulated,
injecting the current along the entire bottom side of those cells, cf. \Cref{fig:geometry-3D}. The source terms $f_i$ are set to zero, and the simulations used homogeneous Dirichlet boundary conditions.

The coupled PDE-ODE system \eqref{eq:EMI_many} was solved by Godunov splitting, where the ODE step
used adaptive time-step integrators from LSODA ~\cite{petzold1983automatic} while the linear system
due to backward Euler discretization of the PDE part was solved by preconditioned conjugate gradient (PCG)
solver with the AMG preconditioner analogous to \eqref{eq:poincare_a}. The PCG solver used the solution from the
previous time step as an initial guess and an absolute error tolerance of $10^{-10}$ as the convergence criterion. 
Due to this setting, the time evolution of PCG iteration count closely follows the stimulus \eqref{eq:stimulus}.
The absence of stimulus at the beginning of the simulation ($t < \units{5}{ms}$) results in a steady state and, in turn, immediate convergence
from the previous (initial) state.  The onset of the stimulus ($t = \units{5}{ms}$) leads to a jump in the iteration count
as well, cf. \Cref{fig:iterations-3D}. We note that, in line with the observations in \Cref{ex:mg},
the time-step size bears little effect on the iterations.

Excitation of the tissue due to applied stimulus is visualized in \Cref{fig:geometry-3D} and
\Cref{fig:iterations-3D}. The figures use a time step of $\tau = \units{0.01}{ms}$, though no differences were observed in the solution plots depending on the time step. 
The excitation of the stimulated second and fourth bottom cells spread through the gap junctions to the entire cell sheet (seen to the right in \Cref{fig:geometry-3D}), leading to a steep rise in the membrane potential of the cells. This characteristic sudden but prolonged rise in the membrane potential  
(as seen to the right in \Cref{fig:iterations-3D})
illustrates the onset of a cardiac action potential. 

\begin{figure}[H]
\includegraphics[height=0.31\textwidth]{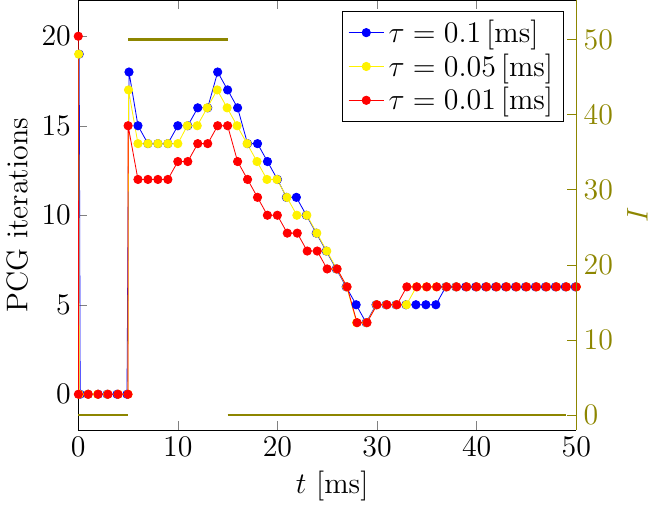}
\hspace{5pt}
\includegraphics[height=0.31\textwidth]{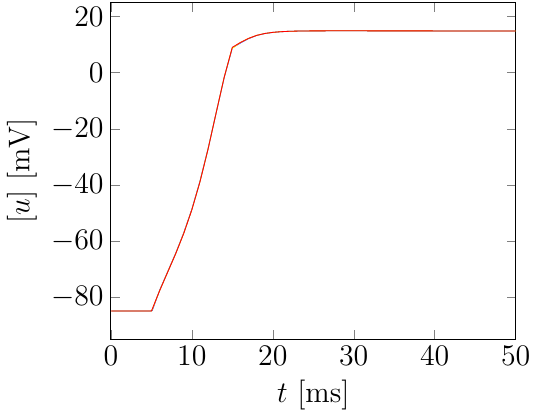}
\caption{
        EMI simulations of 15 cardiac cells.
(Left) History of preconditioned conjugate iterations compared
to the time-dependent stimulus (plotted in green against the right $y$-axis).
(Right) Evolution of the transmembrane potential sampled at the lower-left cell
in \Cref{fig:geometry-3D}.
}
\label{fig:iterations-3D}
\end{figure}

\end{example}

\section{Conclusions} 
The analysis of an interior penalty discontinuous Galerkin method applied to the EMI or cell-by-cell model is presented.  In particular, we show well-posedness of the semi-discrete formulation and of a backward Euler discretization. Error estimates under low spatial regularity are also established. We complement our analysis with several numerical examples where   robust preconditioned iterative solvers are proposed and studied. 

\bibliographystyle{siam}
\bibliography{references}

\begin{thebibliography}{10}

\bibitem{agudelo2013computationally}
{\sc A.~Agudelo-Toro and A.~Neef}, {\em Computationally efficient simulation of electrical activity at cell membranes interacting with self-generated and externally imposed electric fields}, Journal of Neural Engineering, 10 (2013), p.~026019.

\bibitem{aliev_simple_1996}
{\sc R.~R. Aliev and A.~V. Panfilov}, {\em A simple two-variable model of cardiac excitation}, Chaos, Solitons {\&} Fractals, 7 (1996), pp.~293--301.

\bibitem{amar2005existence}
{\sc M.~Amar, D.~Andreucci, P.~Bisegna, and R.~Gianni}, {\em Existence and uniqueness for an elliptic problem with evolution arising in electrodynamics}, Nonlinear Analysis: Real World Applications, 6 (2005), pp.~367--380.

\bibitem{anastassiouephaptic2011}
{\sc C.~A. Anastassiou, R.~Perin, H.~Markram, and C.~Koch}, {\em Ephaptic coupling of cortical neurons}, Nature Neuroscience, 14 (2011), pp.~217--223.

\bibitem{pyamg2023}
{\sc N.~Bell, L.~N. Olson, J.~Schroder, and B.~Southworth}, {\em {PyAMG}: Algebraic multigrid solvers in {P}ython}, Journal of Open Source Software, 8 (2023), p.~5495.

\bibitem{Boffi:book}
{\sc D.~Boffi, F.~Brezzi, and M.~Fortin}, {\em Mixed {F}inite {E}lement {M}ethods and {A}pplications}, vol.~44 of Springer {S}eries in {C}omputational {M}athematics, Springer, 2013.

\bibitem{botta2024high}
{\sc F.~Botta, M.~Calaf{\`a}, P.~C. Africa, C.~Vergara, and P.~F. Antonietti}, {\em High-order discontinuous {G}alerkin methods for the monodomain and bidomain models}, arXiv preprint arXiv:2406.03045,  (2024).

\bibitem{brenner2003poincare}
{\sc S.~C. Brenner}, {\em Poincar{\'e}--{F}riedrichs inequalities for piecewise {$H^1$} functions}, SIAM Journal on Numerical Analysis, 41 (2003), pp.~306--324.

\bibitem{buccino2019does}
{\sc A.~P. Buccino, M.~Kuchta, K.~H. J{\ae}ger, T.~V. Ness, P.~Berthet, K.-A. Mardal, G.~Cauwenberghs, and A.~Tveito}, {\em How does the presence of neural probes affect extracellular potentials?}, Journal of Neural Engineering, 16 (2019), p.~026030.

\bibitem{budivsa2024algebraic}
{\sc A.~Budi{\v{s}}a, X.~Hu, M.~Kuchta, K.-A. Mardal, and L.~Zikatanov}, {\em Algebraic multigrid methods for metric-perturbed coupled problems}, SIAM Journal on Scientific Computing, 46 (2024), pp.~A1461--A1486.

\bibitem{cai2011discontinuous}
{\sc Z.~Cai, X.~Ye, and S.~Zhang}, {\em Discontinuous {G}alerkin finite element methods for interface problems: a priori and a posteriori error estimations}, SIAM Journal on Numerical Analysis, 49 (2011), pp.~1761--1787.

\bibitem{cangiani2018adaptive}
{\sc A.~Cangiani, E.~Georgoulis, and Y.~Sabawi}, {\em Adaptive discontinuous {G}alerkin methods for elliptic interface problems}, Mathematics of Computation, 87 (2018), pp.~2675--2707.

\bibitem{ciarlet2013analysis}
{\sc P.~Ciarlet}, {\em Analysis of the {S}cott--{Z}hang interpolation in the fractional order {S}obolev spaces}, Journal of Numerical Mathematics, 21 (2013), pp.~173--180.

\bibitem{ciarlet2013linear}
{\sc P.~G. Ciarlet}, {\em Linear and nonlinear functional analysis with applications}, vol.~130, SIAM, 2013.

\bibitem{carry2021abstractions}
{\sc C.~Daversin-Catty, C.~N. Richardson, A.~J. Ellingsrud, and M.~E. Rognes}, {\em Abstractions and automated algorithms for mixed domain finite element methods}, ACM Trans. Math. Softw., 47 (2021).

\bibitem{di2011mathematical}
{\sc D.~A. Di~Pietro and A.~Ern}, {\em Mathematical aspects of discontinuous {G}alerkin methods}, vol.~69, Springer Science {\&} Business Media, 2011.

\bibitem{ellingsrud2024splitting}
{\sc A.~J. Ellingsrud, P.~Benedusi, and M.~Kuchta}, {\em A splitting, discontinuous {G}alerkin solver for the cell-by-cell electroneutral {N}ernst-{P}lanck framework}, arXiv preprint arXiv:2404.08320,  (2024).

\bibitem{ern2004theory}
{\sc A.~Ern and J.-L. Guermond}, {\em Theory and Practice of Finite Elements}, vol.~159, Springer, 2004.

\bibitem{ern2016mollification}
\leavevmode\vrule height 2pt depth -1.6pt width 23pt, {\em Mollification in strongly {L}ipschitz domains with application to continuous and discrete de {R}ham complexes}, Computational Methods in Applied Mathematics, 16 (2016), pp.~51--75.

\bibitem{ern2017finite}
\leavevmode\vrule height 2pt depth -1.6pt width 23pt, {\em Finite element quasi-interpolation and best approximation}, ESAIM: Mathematical Modelling and Numerical Analysis, 51 (2017), pp.~1367--1385.

\bibitem{Ern:booki}
{\sc A.~Ern and J.~L. Guermond}, {\em Finite {E}lements {I}}, Texts in {A}pplied {M}athematics, Springer International Publishing, 2021.

\bibitem{ern2021quasi}
{\sc A.~Ern and J.-L. Guermond}, {\em Quasi-optimal nonconforming approximation of elliptic {PDE}s with contrasted coefficients and {$H^{1+r}$}, {$r > 0$}, regularity}, Foundations of Computational Mathematics, 22 (2021), pp.~1273--1308.

\bibitem{hypre}
{\sc R.~D. Falgout and U.~M. Yang}, {\em hypre: A library of high performance preconditioners}, in Computational Sciences --- ICCS 2002, P.~M.~A. Sloot, A.~G. Hoekstra, C.~J.~K. Tan, and J.~J. Dongarra, eds., Berlin, Heidelberg, 2002, Springer Berlin Heidelberg, pp.~632--641.

\bibitem{fokoue2023numerical}
{\sc D.~Fokou{\'e} and Y.~Bourgault}, {\em Numerical analysis of finite element methods for the cardiac extracellular-membrane-intracellular model: {S}teklov--{P}oincar{\'e} operator and spatial error estimates}, ESAIM: Mathematical Modelling and Numerical Analysis, 57 (2023), pp.~2595--2621.

\bibitem{franzone2002degenerate}
{\sc P.~C. Franzone and G.~Savar{\'e}}, {\em Degenerate evolution systems modeling the cardiac electric field at micro-and macroscopic level}, in Evolution Equations, Semigroups and Functional Analysis: In memory of Brunello Terreni, Springer, 2002, pp.~49--78.

\bibitem{grosberg2011ensembles}
{\sc A.~Grosberg, P.~W. Alford, M.~L. McCain, and K.~K. Parker}, {\em Ensembles of engineered cardiac tissues for physiological and pharmacological study: Heart on a chip}, Lab on a Chip, 11 (2011), pp.~4165--4173.

\bibitem{Gudi2010ANE}
{\sc T.~Gudi}, {\em A new error analysis for discontinuous finite element methods for linear elliptic problems}, Math. Comput., 79 (2010), pp.~2169--2189.

\bibitem{hiptmair2006operator}
{\sc R.~Hiptmair}, {\em Operator preconditioning}, Computers \& Mathematics with Applications, 52 (2006), pp.~699--706.

\bibitem{huang2020high}
{\sc H.~Huang, J.~Li, and J.~Yan}, {\em High order symmetric direct discontinuous {G}alerkin method for elliptic interface problems with fitted mesh}, Journal of Computational Physics, 409 (2020), p.~109301.

\bibitem{huynh2013high}
{\sc L.~Huynh, N.~C. Nguyen, J.~Peraire, and B.~C. Khoo}, {\em A high-order hybridizable discontinuous {G}alerkin method for elliptic interface problems}, International Journal for Numerical Methods in Engineering, 93 (2013), pp.~183--200.

\bibitem{huynh2023convergence}
{\sc N.~M.~M. Huynh, F.~Chegini, L.~F. Pavarino, M.~Weiser, and S.~Scacchi}, {\em Convergence analysis of {BDDC} preconditioners for composite {DG} discretizations of the cardiac cell-by-cell model}, SIAM Journal on Scientific Computing, 45 (2023), pp.~A2836--A2857.

\bibitem{jaeger2019properties}
{\sc K.~H. J{\ae}ger, A.~G. Edwards, A.~McCulloch, and A.~Tveito}, {\em Properties of cardiac conduction in a cell-based computational model}, PLoS Computational Biology, 15 (2019), p.~e1007042.

\bibitem{jaeger2021efficient}
{\sc K.~H. J{\ae}ger, K.~G. Hustad, X.~Cai, and A.~Tveito}, {\em Efficient numerical solution of the {EMI} model representing the extracellular space ({E}), cell membrane ({M}) and intracellular space ({I}) of a collection of cardiac cells}, Frontiers in Physics, 8 (2021), p.~579461.

\bibitem{doi:10.1137/S0036142902405217}
{\sc O.~A. Karakashian and F.~Pascal}, {\em A posteriori error estimates for a discontinuous {G}alerkin approximation of second-order elliptic problems}, SIAM Journal on Numerical Analysis, 41 (2003), pp.~2374--2399.

\bibitem{liu2017free}
{\sc H.~Liu and Z.~Wang}, {\em A free energy satisfying discontinuous {G}alerkin method for one-dimensional {P}oisson--{N}ernst--{P}lanck systems}, Journal of Computational Physics, 328 (2017), pp.~413--437.

\bibitem{liu2022positivity}
{\sc H.~Liu, Z.~Wang, P.~Yin, and H.~Yu}, {\em Positivity-preserving third order {DG} schemes for {P}oisson--{N}ernst--{P}lanck equations}, Journal of Computational Physics, 452 (2022), p.~110777.

\bibitem{mardal2011preconditioning}
{\sc K.-A. Mardal and R.~Winther}, {\em Preconditioning discretizations of systems of partial differential equations}, Numerical Linear Algebra with Applications, 18 (2011), pp.~1--40.

\bibitem{petzold1983automatic}
{\sc L.~Petzold}, {\em Automatic selection of methods for solving stiff and nonstiff systems of ordinary differential equations}, SIAM Journal on Scientific and Statistical Computing, 4 (1983), pp.~136--148.

\bibitem{riviere2008discontinuous}
{\sc B.~Rivi{\`e}re}, {\em Discontinuous {G}alerkin Methods for Solving Elliptic and Parabolic Equations: Theory and Implementation}, SIAM, 2008.

\bibitem{rohr2004role}
{\sc S.~Rohr}, {\em Role of gap junctions in the propagation of the cardiac action potential}, Cardiovascular research, 62 (2004), pp.~309--322.

\bibitem{roy2023scalable}
{\sc T.~Roy, J.~Andrej, and V.~A. Beck}, {\em A scalable {DG} solver for the electroneutral {N}ernst-{P}lanck equations}, Journal of Computational Physics, 475 (2023), p.~111859.

\bibitem{amg}
{\sc J.~W. Ruge and K.~St{\"u}ben}, {\em 4. Algebraic Multigrid}, SIAM, 1987, ch.~4, pp.~73--130.

\bibitem{schlegel2024whole}
{\sc P.~Schlegel, Y.~Yin, A.~S. Bates, S.~Dorkenwald, et~al.}, {\em Whole-brain annotation and multi-connectome cell typing of {D}rosophila}, Nature, 634 (2024), pp.~139--152.

\bibitem{scott1990finite}
{\sc L.~R. Scott and S.~Zhang}, {\em Finite element interpolation of nonsmooth functions satisfying boundary conditions}, Mathematics of Computation, 54 (1990), pp.~483--493.

\bibitem{sundnes2007computing}
{\sc J.~Sundnes, G.~T. Lines, X.~Cai, B.~F. Nielsen, K.-A. Mardal, and A.~Tveito}, {\em Computing the electrical activity in the heart}, vol.~1, Springer Science {\&} Business Media, 2007.

\bibitem{tveito2017cell}
{\sc A.~Tveito, K.~H. J{\ae}ger, M.~Kuchta, K.-A. Mardal, and M.~E. Rognes}, {\em A cell-based framework for numerical modeling of electrical conduction in cardiac tissue}, Frontiers in Physics, 5 (2017).

\bibitem{tveito2021modeling}
{\sc A.~Tveito, K.-A. Mardal, and M.~E. Rognes}, {\em Modeling Excitable Tissue: The EMI Framework}, Springer Nature, 2021.

\bibitem{samg}
{\sc P.~Vanek, J.~Mandel, and M.~Brezina}, {\em Algebraic multigrid by smoothed aggregation for second and fourth order elliptic problems}, Computing, 56 (1996), pp.~179--196.

\bibitem{veneroni2006reaction}
{\sc M.~Veneroni}, {\em Reaction--diffusion systems for the microscopic cellular model of the cardiac electric field}, Mathematical Methods in the Applied Sciences, 29 (2006), pp.~1631--1661.

\end{thebibliography}
\appendix 
\section{Details for the interface equality \eqref{eq:interface_boundary_terms}}
\label{appendix:interface_cond} 
We focus on the case  $F \in \mathcal{F}_{\Gamma,h}$ with $F = \partial K_i \cap \partial K_e$ ($K_i \subset \Omega_i)$. The case $F \subset \Gamma_e$ is handled similarly.  We write
\begin{align} \label{eq;interface_expansion}
& \int_{F} [w_h \mathcal K_{\delta}^d(\bm \sigma(u))] \cdot \bm n_F  \\&  =  \nonumber \int_{F} (w_h \mathcal K_{\delta}^d(\bm \sigma (u)
))_{\vert K_e} \cdot \bm n_{K_e}  +  \int_{F} (w_h \mathcal K_{\delta}^d(\bm \sigma (u)))_{\vert K_i} \cdot \bm n_{K_i} 
\end{align}
The next steps focus on the first term above. The second term is handled in the same manner. 

\textit{Step 1}. Show that \begin{equation}
    \lim_{\delta \rightarrow 0} \int_{F} (w_h \mathcal K_{\delta}^d(\bm \sigma (u)))_{\vert K_e} \cdot \bm n_{K_e}  =\langle  \bm \sigma(u) \cdot \bm n_{K_e}, w_h{}_{\vert K_e} \rangle_{F} \label{eq:appendix_step1}
\end{equation}
Indeed, since $\mathcal K_{\delta}^d(\bm \sigma (u))$ is smooth, Green's theorem allows us to write 
\begin{equation} \label{eq:lim_Kdelta_F}
    \lim_{\delta \rightarrow 0} \int_{F} (w_h \mathcal K_{\delta}^d(\bm \sigma (u)
))_{\vert K_e} \cdot \bm n_{K_e} =  \lim_{\delta \rightarrow 0} \langle \mathcal K_{\delta}^d(\bm \sigma (u))
  \cdot \bm n_{K_e} , w_h{}_{\vert K_e} \rangle_{F}.   
\end{equation}
Considering the definition \eqref{eq:def_duality_pair}, the commutativity property \eqref{eq:commutative_prop_moll} and \cite[Theorem 3.3]{ern2016mollification} on the convergence of the mollification operators, we obtain  
\begin{multline} \label{eq:lim_delta_facet}
\lim_{\delta \rightarrow 0} \langle  (\mathcal K_{\delta}^d(\bm \sigma (u)) - \bm \sigma(u) )
  \cdot \bm n_{K_e}, w_h{}_{\vert K_e} \rangle_{F}   \\ = \lim_{\delta \rightarrow 0 } \int_{K_e}  ( \mathcal K_{\delta}^d(\bm \sigma (u)) - \bm \sigma(u) ) \cdot \nabla L_F^{K_e}( w_h{}_{\vert K_e})   \\  + \lim_{\delta \rightarrow 0 } \int_{K_e}
(\mathcal K_{\delta}^b ( \nabla \cdot \bm \sigma (u)) -  \nabla \cdot \bm \sigma(u) ) L_F^{K_e}( w_h{}_{\vert K_e})  = 0. 
\end{multline}
Using \eqref{eq:lim_delta_facet} in \eqref{eq:lim_Kdelta_F} yields \eqref{eq:appendix_step1}. 

\textit{Step 2}.  
Show that 
\begin{equation} \label{eq:interface_condition}
\langle  \bm \sigma(u) \cdot \bm n_{K_e}, w_h{}_{\vert K_e} \rangle_{F}    =  -  \int_{F}  \left( C_M \frac{\partial [u]}{\partial t} + f([u])\right)  w_h{}_{\vert K_e}. 
\end{equation}
To this end, let $L_F^{\Omega_e} (w_h) = L_F^{K_e} (w_h{}_{\vert K_e}) $ on $K_e$ and $0$ otherwise. From the properties of $L_F^{K_e}$, it is clear that $L_F^{\Omega_e} (w_h)$ belongs to the broken $W^{1,\rho'}(\mesh \cap \Omega_e)$ space and that $[L_F^{\Omega_e} (w_h)] = 0$ on $F \in \mathcal{F}_{e,h}$. Thus, from \cite[Lemma 1.23]{di2011mathematical}, it follows that $L_F^{\Omega_e} (w_h)\in W^{1,\rho'} (\Omega_e)$.  Consider 
\begin{align}
\langle  \bm \sigma(u) \cdot \bm n_{K_e}, w_h{}_{\vert K_e} \rangle_{F}   & =   \int_{K_e}  (\bm \sigma(u)  \cdot \nabla L_F^{K_e}(w_h) + \nabla \cdot (\bm \sigma(u) ) L_F^{K_e}(w_h)  )\mathrm{d}x \\ 
& = \int_{\Omega_e}  (\bm \sigma(u)  \cdot \nabla L_F^{\Omega_e}(w_h) + f_e L_F^{\Omega_e}(w_h)  )\mathrm{d}x,  \nonumber 
\end{align}
where we used that $\nabla \cdot (\bm \sigma (u) ) = - \nabla(\kappa_e \nabla u_e) = f_e$ in $L^2(\Omega_e)$. 
Since $\Omega_e$ is a Lipschitz domain, we use density of smooth functions $C^\infty(\overline{\Omega_e})$ in $W^{1,\rho'}(\Omega_e)$. Namely, let $\{w_m\} \in C^\infty(\overline{\Omega_e}) $ converge to $L_F^{\Omega_e}(w_h)$ in $W^{1,\rho'}(\Omega_e)$. Since $\bm \sigma(u) \in L^\rho (\Omega)$, $f  \in L^2(\Omega)$ and $W^{1,\rho'}(\Omega_e) \hookrightarrow L^2(\Omega_e)$ \cite[Theorem 2.31]{Ern:booki}, we can write the above as
\begin{align}
\langle  \bm \sigma(u) \cdot \bm n_{K_e}, w_h{}_{\vert K_e} \rangle_{F}  
& =  \lim_{m \rightarrow \infty} \int_{\Omega_e}  (\bm \sigma(u)  \cdot \nabla w_m + f  w_m  )\mathrm{d}x  \\ 
& = -   \lim_{m \rightarrow \infty}  \int_{\Gamma }  \left( C_M \frac{\partial [u]}{\partial t} + f_{\Gamma}([u])\right)  w_m{}_{\vert \Omega_e} \mathrm{d}s. \nonumber
\end{align}
The last equality follows by testing \eqref{eq:EMI_weak} with $v = w_m $ in $\Omega_e$ and $v = 0 $ in $\Omega_i$. 
Now note that from trace theorem, see for e.g.  \cite[Theorem 3.15]{Ern:booki}, the Dirichlet trace operator is bounded as a  map from  $W^{1,q}(D) \rightarrow L^{q}(\Gamma)$ for $D = \Omega_i$ or $D = \Omega_e$ and $q \in \{\rho, \rho'\}$ since $\rho, \rho' >1$. Consequently, we know that $ w_m{}_{\vert \Omega_e} \rightarrow L_F^{\Omega_e} (w_h)_{\vert \Omega_e}    \in  L^{\rho'}(\Gamma)$. Since by assumption $\partial_t u \in H^{1/2+s}(\Omega_e \cup \Omega_i)$, we have that $\gamma ( [\partial_t u]) \in H^{s}(\Gamma)$ which in turn implies $\gamma( \partial_t[ u ]) \in L^{\rho}(\Gamma)$. Along with the Lipschitz continuity of  $f_{\Gamma}$, this implies that $\frac{\partial [u]}{\partial t} + f_{\Gamma}([u]) \in L^{\rho}(\Gamma)$.  Thus, the above limit evaluates to 
\begin{align}\label{eq:evaluate_limit_appendix}
\langle  \bm \sigma(u) \cdot \bm n_{K_e}, w_h{}_{\vert K_e} \rangle_{F}  & = - \int_{\Gamma }  \left( C_M \frac{\partial [u]}{\partial t} + f_{\Gamma}([u])\right)  L_F^{\Omega_e}(w_h){}_{\vert \Omega_e} ds \\ 
& =    -\int_{F}  \left( C_M \frac{\partial [u]}{\partial t} + f([u])\right)  w_h{}_{\vert K_e}.  \nonumber 
\end{align}
The last equality follows since $L_F^{\Omega_e} (w_h) \vert_{\Gamma} $ is nonzero only on $F$.  This shows \eqref{eq:interface_condition}. 

\textit{Step 3}. Combining steps 1 and 2

Using  \eqref{eq:interface_condition} in \eqref{eq:appendix_step1} yields 
\begin{align}
\lim_{\delta \rightarrow 0} \int_{F} (w_h \mathcal K_{\delta}^d(\bm \sigma (u)))_{\vert K_e} \cdot \bm n_{K_e} & =\langle  \bm \sigma(u) \cdot \bm n_{K_e}, w_h{}_{\vert K_e} \rangle_{F} \\ &=   - \int_{F}  \left( C_M \frac{\partial [u]}{\partial t} + f([u])\right)  w_h{}_{\vert K_e}.  \nonumber
\end{align}
Proceeding similarly for the second term in \eqref{eq;interface_expansion} yields 
\begin{align}
\lim_{\delta \rightarrow 0 } \int_{F} [w_h \mathcal K_{\delta}^d(\bm \sigma(u))] \cdot \bm n_F  = \int_{F}  \left( C_M \frac{\partial [u]}{\partial t} + f([u])\right)  [w_h] .
\end{align}
We use  the same arguments for $F \subset \Gamma_e $ and note that (in this case) the expression in \eqref{eq:evaluate_limit_appendix} evaluates to zero.  With summing over the faces in $\mathcal{F}_{\Gamma,h} \cup  \Gamma_e$, one obtains \eqref{eq:interface_boundary_terms}. 
\end{document}